\documentclass[12pt]{amsart}
\usepackage{amsmath}
\usepackage{amssymb}
\usepackage{amsthm}
\usepackage{amscd}
\usepackage{enumerate}
\usepackage{enumitem}
\usepackage{verbatim}
\usepackage[all]{xy}
\usepackage{mathrsfs}
\usepackage{mathabx}
\usepackage{url}

\theoremstyle{plain}
\newtheorem{thm}{Theorem}[section]
\newtheorem{prop}[thm]{Proposition}
\newtheorem{lem}[thm]{Lemma}
\newtheorem{cor}[thm]{Corollary}

\theoremstyle{definition}
\newtheorem{dfn}[thm]{Definition}
\newtheorem{rmk}[thm]{Remark}

\newcommand{\diag}{\mathrm{diag}}

\newcommand{\End}{\mathrm{End}}

\newcommand{\har}{\mathrm{har}}

\newcommand{\Hom}{\mathrm{Hom}}

\newcommand{\ord}{\mathrm{ord}}

\newcommand{\Res}{\mathrm{Res}}

\newcommand{\st}{\mathrm{st}}
\newcommand{\St}{\mathrm{St}}
\newcommand{\src}{\mathrm{src}}

\newcommand{\Sym}{\mathrm{Sym}}

\newcommand{\Kbar}{\bar{K}}

\newcommand{\cL}{\mathcal{L}}

\newcommand{\cO}{\mathcal{O}}

\newcommand{\cT}{\mathcal{T}}

\newcommand{\cV}{\mathcal{V}}

\newcommand{\vep}{\varepsilon}

\newcommand{\bC}{\mathbb{C}}

\newcommand{\bF}{\mathbb{F}}

\newcommand{\bP}{\mathbb{P}}
\newcommand{\bQ}{\mathbb{Q}}

\newcommand{\bZ}{\mathbb{Z}}

\newcommand{\frB}{\mathfrak{B}}

\newcommand{\frem}{\mathfrak{m}}
\newcommand{\frn}{\mathfrak{n}}

\makeatletter
    
    \@addtoreset{equation}{section}
\makeatother

\makeatletter
\renewcommand{\p@enumii}{}
\makeatother

\begin{document}

\title[$\wp$-adic continuous families of Drinfeld eigenforms]{$\wp$-adic continuous families of Drinfeld 
eigenforms of finite 
slope}
\author{Shin Hattori}
\address[Shin Hattori]{Department of Natural Sciences, Tokyo City University}

\date{\today}


\begin{abstract}
Let $p$ be a rational prime, $v_p$ the normalized $p$-adic valuation on $\bZ$, $q>1$ a $p$-power and $A=\bF_q[t]$. Let $\wp\in A$ be an irreducible polynomial 
and $\frn
\in A$ a non-zero element which is prime to $\wp$. 
Let $k\geq 2$ and $r\geq 1$ be integers. We denote by $S_k(\Gamma_1(\frn\wp^r))$ the space of Drinfeld cuspforms of level $\Gamma_1(\frn\wp^r)$ and weight $k$ 
for $\bF_q(t)$. 
Let $n\geq 1$ be an integer and $a\geq 0$ a rational number. 
Suppose that $\frn\wp$ has a prime factor of degree one and the generalized eigenspace in $S_k(\Gamma_1(\frn\wp^r))$ of slope $a$ is one-dimensional.
In this paper, under an assumption that $a$ is sufficiently small, 
we construct a family $\{F_{k'}\mid v_p(k'-k)\geq \log_p(p^n+a)\}$ of Hecke eigenforms $F_{k'}\in S_{k'}(\Gamma_1(\frn\wp^r))$ of slope $a$ such that, for any 
$Q\in 
A$, 
	the Hecke eigenvalues of $F_k$ and $F_{k'}$ at $Q$ are congruent modulo $\wp^\kappa$ with some $\kappa>p^{v_p(k'-k)}-p^n-a$.
\end{abstract}

\maketitle



\section{Introduction}\label{SecIntro}



Let $p$ be a rational prime, $q>1$ a $p$-power and $\bF_q$ the field of $q$ elements. 
Put $A=\bF_q[t]$ and $K=\bF_q(t)$. Let $\wp\in A$ be an irreducible polynomial of positive degree, 
$\frn$ a non-zero element of $A$ which is prime to $\wp$ and $r\geq 1$ an integer. 
Put $A_r=A/(\wp^r)$ and $\kappa(\wp)=A/(\wp)$.
We 
denote by $K_\wp$ the $\wp$-adic completion of $K$, by $\bC_\wp$ the $\wp$-adic completion of an algebraic closure of $K_\wp$ and 
by $v_\wp:\bC_\wp\to \bQ\cup\{+\infty\}$ the $\wp$-adic 
additive valuation on $\bC_\wp$ normalized as $v_\wp(\wp)=1$.
Similarly, we denote by $K_\infty$ the $(1/t)$-adic completion of $K$ and by $\bC_\infty$ the $(1/t)$-adic completion of an 
algebraic closure of $K_\infty$. 
Let $\Kbar$ be the algebraic closure of $K$ inside $\bC_\infty$ and we fix an embedding of $K$-algebras $\iota_\wp: \bar{K}\to \bC_\wp$. 
For any $x\in \bar{K}$, we define its normalized $\wp$-adic valuation by $v_\wp(\iota_\wp(x))$.
Let $\Omega=\bP^1(\bC_\infty)\setminus \bP^1(K_\infty)$ be the Drinfeld upper half plane, which has a natural structure of a 
rigid 
analytic variety over $K_\infty$.

Let $\Gamma$ be a subgroup of $\mathit{SL}_2(A)$ and $k$ an integer. A Drinfeld modular form of level $\Gamma$ and weight $k$ is a rigid analytic function 
on $\Omega$ satisfying
\[
f\left(\frac{az+b}{cz+d}\right)=(cz+d)^kf(z)\quad \text{for any }\gamma=\begin{pmatrix}
a&b\\c&d
\end{pmatrix}\in \Gamma, z\in \Omega
\]
and a holomorphy condition at cusps. It is considered as a function field analogue of the notion of elliptic modular form. 

Recently, $\wp$-adic properties of Drinfeld modular forms have attracted attention and have been studied actively (for example, 
\cite{BV1,BV0,BV_slope,Goss_v,Ha_HTT, Ha_GVt,PZ,Vincent}). 
However, though we have a highly 
developed theory of $p$-adic analytic families of elliptic eigenforms of finite slope, $\wp$-adic properties of Drinfeld modular forms are much less 
well-understood
compared to the elliptic case. 
One of the difficulties in the Drinfeld case is that,
since the group $\cO_{K_\wp}^\times$ is topologically of infinitely generated, analogues of the completed group ring $\bZ_p[[\bZ_p^\times]]$ 
are not Noetherian, and it seems that we have no good definition of characteristic power series applicable to non-Noetherian base rings, as mentioned in 
\cite[paragraph
before Lemma 
2.3]{Buz_EV}.

In this paper, we will construct families of Drinfeld eigenforms in which Hecke eigenvalues vary in a $\wp$-adically continuous way. For the precise statement, 
we fix some notation. For any $\frem\in A$, we put
\[
\Gamma_1(\frem)=\left\{\gamma \in \mathit{SL}_2(A)\ \middle|\ \gamma\equiv \begin{pmatrix}
1&*\\0&1
\end{pmatrix}\bmod \frem\right\}.
\]
Let $\Theta$ be any subgroup of $1+\wp A_r\subseteq A_r^\times$. We define
\[
\Gamma_0^\Theta(\wp^r)=\left\{\gamma \in \mathit{SL}_2(A)\ \middle|\ \gamma\bmod \wp^r \in \begin{pmatrix}
\Theta&*\\0&\Theta
\end{pmatrix}\right\}\subseteq \Gamma_1(\wp)
\]
and $\Gamma_1^\Theta(\frn,\wp^r)=\Gamma_1(\frn)\cap \Gamma_0^\Theta(\wp^r)$, which satisfies $\Gamma_1^{\{1\}}(\frn,\wp^r)=\Gamma_1(\frn\wp^r)$.

Let $k\geq 2$ be an integer. For any non-zero element $Q\in A$, the Hecke operator $T_Q$ acts on 
the $\bC_\infty$-vector space $S_k(\Gamma_1^\Theta(\frn,\wp^r))$ of Drinfeld cuspforms of level $\Gamma_1^\Theta(\frn,\wp^r)$ and weight $k$. 
The operator $T_\wp$ is also denoted by $U$. 
Since they stabilize an $A$-lattice $\cV_k(A)$ (Proposition \ref{PropInteg}), every eigenvalue of $T_Q$ is integral over $A$. 
The normalized $\wp$-adic valuation of an eigenvalue of $U$ is called slope, and we denote by $d(k,a)$ the dimension of 
the generalized $U$-eigenspace for the eigenvalues of slope $a$. For any Hecke eigenform $F$, its $T_Q$-eigenvalue is denoted by $\lambda_Q(F)$. 
We denote by $v_p$ the $p$-adic valuation on $\bZ$ satisfying $v_p(p)=1$.
Then the main theorem of this paper (Theorem \ref{ThmBF}) gives the following, which we will prove in \S\ref{SubsecBF}.

\begin{thm}\label{ThmMainIntro}
	Suppose that $\frn\wp$ has a prime factor $\pi$ of degree one. Let $n\geq 1$ and $k\geq 2$ be integers. 
	Put $d=[\Gamma_1(\pi):\Gamma_1^\Theta(\frn,\wp^r)]$, $\vep=d(k,0)$ and
	\begin{align*}
	D_2(n,d,\vep)&=\frac{1}{d}\left\{\sqrt{2dp^n+(d-\vep+1)(2d-\vep-1)}-\frac{3}{2}d+\vep\right\},\\
	D(n,d,\vep)&=\min\left\{p^n\left(\frac{4+d p^n -d}{4+2dp^n-2\vep}\right), D_2(n,d,\vep)\right\}.
	\end{align*}
	Let $a$ be any non-negative rational number satisfying 
	\[
	a<\min\{D(n,d,\vep),k-1\}.
	\]
	
	Suppose $d(k,a)=1$. Then, for any integer $k'\geq k$ satisfying 
	\[
	v_p(k'-k)\geq \log_p(p^n+a), 
	\]
	there exists a Hecke eigenform $F_{k'}\in S_{k'}(\Gamma_1^\Theta(\frn,\wp^r))$ of slope $a$
	such that for any $Q$ we have
	\[
	v_\wp(\iota_\wp(\lambda_Q(F_{k'})-\lambda_Q(F_{k})))> p^{v_p(k'-k)}-p^n-a.
	\]
\end{thm}
In fact, what we will prove allow nebentypus characters at $\wp$ (Remark \ref{RmkMain0}).

For example, in the case of $\frn=1$, $\wp=t$ and $r=1$, we have $\Gamma_1^\Theta(\frn,\wp^r)=\Gamma_1(t)$, $d=\vep=1$ and $D(n,1,1)=\sqrt{2p^n}-\frac{1}{2}$.
In this case, Theorem \ref{ThmMainIntro} implies that, for any Hecke eigenform $F_k$ of slope zero in $S_k(\Gamma_1(t))$,
the $T_Q$-eigenvalue $\lambda_Q(F_k)$ is $t$-adically arbitrarily close to those coming from Hecke eigenforms with $A$-expansion \cite{Pet},
which shows $\lambda_Q(F_k)=1$ for any $Q$ (Proposition \ref{PropOrdt}). 
This suggests that, though we will prove constancy results of the dimension of slope zero cuspforms with respect to $k$ and $r$ (Proposition 
\ref{PropOrdAgree} and 
Proposition \ref{PropHida}), Hida theory for the level $\Gamma_0(t^r)$ should be trivial (Remark \ref{RmkHida}).
We also note that families constructed in Theorem \ref{ThmMainIntro} contain Hecke eigenforms whose Hecke eigenvalue at $Q$ is not a power of $Q$ (\S
\ref{SubsecEx}), 
and thus they capture a more subtle $\wp$-adic structure of Hecke eigenvalues than the theory of $A$-expansions.

Let us explain the idea of the proof of Theorem \ref{ThmMainIntro}. Note that a usual method to construct $p$-adic families of eigenforms of finite slope in 
the 
number field 
case is the use of the Riesz theory \cite{Col_Banach,Buz_EV}, which is not available for our case at present, due to the lack of a notion of characteristic 
power series over non-Noetherian Banach algebras. Instead, we follow an idea of Buzzard \cite{Buz_DQA} by which he constructed $p$-adically continuous families 
of quaternionic 
eigenforms 
over $\bQ$.

First we will prove a variant of the Gouv\^{e}a-Mazur conjecture (Proposition \ref{PropPerturb}),
which implies $d(k,a)=d(k',a)$ if $k$ and $k'$ are highly congruent $p$-adically and $a$ is 
sufficiently small. With the assumption $d(k,a)=1$, it produces Hecke eigenforms $F_k$ and $F_{k'}$ of slope $a$ in weights $k$ and $k'$, 
respectively. 
For this part, 
we employ the same idea as in \cite{Ha_GVt}: a lower bound of elementary divisors of the representing matrix of $U$ with some basis 
and a perturbation lemma \cite[Theorem 4.4.2]{Ked} yield the equality. 
To obtain such a bound (Corollary \ref{CorElDiv}), we need to define Hecke operators acting on the 
Steinberg complex 
(\ref{EqnExactVLL}) with 
respect to $\Gamma_1^\Theta(\frn,\wp^r)$, which is done in \S\ref{SubsecHecke}. Note that similar Hecke operators on a Steinberg complex in an adelic setting 
are
given in \cite[\S6.4]{Boeckle}.

Then, a weight reduction map (\S\ref{SubsecWtRed}) yields a Drinfeld cuspform $G$ of weight $k$ such that, for $m=v_p(k'-k)$, the element
 $G\bmod \wp^{p^m}$ is a Hecke eigenform with the same
eigenvalues as those of $F_{k'} \bmod \wp^{p^m}$. Now the point is that, if two lines generated by $F_k$ and $G$ are highly congruent in some 
sense, then we 
can show that the eigenvalues of $F_k$ and $G \bmod \wp^{p^m}$ are also highly congruent, which gives Theorem \ref{ThmMainIntro}; otherwise the two lines are so 
far apart that, again by the Gouv\^{e}a-Mazur variant mentioned above, they produce $U$-eigenvalues of slope $a$ with multiplicity 
more than one, which contradicts 
$d(k,a)=1$ (Theorem \ref{ThmBF}).

\subsection*{Acknowledgements} The author would like to thank Gebhard B\"{o}ckle for suggesting him to look for $\wp$-adically continuous families of Drinfeld 
eigenforms instead of $\wp$-adically analytic ones, and David Goss for a helpful discussion. This work was supported by JSPS KAKENHI Grant Number JP17K05177.



\section{Drinfeld cuspforms via the Steinberg module}\label{SecDCF}


For any arithmetic subgroup $\Gamma$ of $\mathit{SL}_2(A)$ and any integer $k\geq2$, we denote by $S_k(\Gamma)$ the space of Drinfeld cuspforms of level $\Gamma
$ 
and weight $k$.
In this section, we first recall an interpretation of $S_k(\Gamma)$ using the Steinberg module due to Teitelbaum \cite[p.~506]{Tei}, 
following the normalization of \cite[\S5]{Boeckle}. We also introduce Hecke operators acting on the Steinberg complex. Using them, we define an $A$-lattice of 
the space of Drinfeld cuspforms which is 
stable under the Hecke action.


\subsection{Steinberg module}\label{SubsecSt}


For any $A$-algebra $B$, we consider $B^2$ as the set of row vectors, and define a left action $\circ$ of $\mathit{GL}_2(B)$ on it by $\gamma\circ x=x
\gamma^{-1}$. 
Let $\cT$ be the Bruhat-Tits tree for $\mathit{SL}_2(K_\infty)$. We denote by $\cT_0$ the set of vertices of $\cT$, 
which is the set of $K_\infty^\times$-equivalence classes of $\cO_{K_\infty}$-lattices in $K_\infty^2$, 
and by $\cT_1$ the set of its edges. The oriented graph associated with $\cT$ and the 
set of 
oriented edges are denoted by $\cT^o$ and $\cT^o_1$, respectively. For any oriented edge $e$, we denote its origin by $o(e)$, 
its terminus by $t(e)$ and the opposite edge by $-e$. The group $\{\pm 1\}$ acts on $\cT^o_1$ by $(-1)e=-e$.

Let $\Gamma$ be an arithmetic subgroup of $\mathit{SL}_2(A)$ \cite[\S3.4]{Boeckle}, and we assume $\Gamma$ to be $p'$-torsion free (namely, every element of $
\Gamma$ of finite order has $p$-power order). 
The group $
\Gamma$ acts on $\cT$ and $\cT^o$ via the natural inclusion $\Gamma\to \mathit{GL}_2(K_\infty)$. We say a vertex or an oriented edge of $\cT$ is $\Gamma$-stable 
if its stabilizer subgroup in $\Gamma$ is trivial, and $\Gamma$-unstable otherwise. 
We denote by $\cT_0^\st$ and $\cT_1^{o,
\st}$ the subsets of $\Gamma$-stable elements. 
For any $\Gamma$-unstable vertex $v$, its stabilizer subgroup in $\Gamma$ is a non-trivial finite $p$-group and thus fixes a unique rational end which we denote 
by 
$b(v)$ 
\cite[Ch.~II, \S2.9]{Se}.

For any ring $R$ and any set $S$, we write $R[S]$ for the free $R$-module with basis $\{[s]\mid s\in S\}$. When $S$ admits a left action of $\Gamma$, the $R$-
module $R[S]$ 
also admits a natural left action of the group ring $R[\Gamma]$ which we denote by $\circ$. In this case, we also define a right action of $\Gamma$ on $R[S]$ by 
$[s]|_{\gamma}=\gamma^{-1}
\circ [s]$, which makes it a right $R[\Gamma]$-module.

Put
\[
\bZ[\bar{\cT}_1^{o,\st}]=\bZ[\cT_1^{o,\st}]/\langle [e]+[-e]\mid e\in \cT_1^{o,\st}\rangle.
\]
We define a surjection of $\bZ[\Gamma]$-modules 
$\partial_\Gamma: \bZ[\cT_1^{o,\st}]\to \bZ[\cT_0^\st]$ by $\partial_\Gamma(e)=[t(e)]-[o(e)]$, 
where we put $[v]=0$ in $\bZ[\cT_0^\st]$ for any $\Gamma$-unstable vertex $v$. 
It factors as $\partial_\Gamma:\bZ[\bar{\cT}_1^{o,\st}]\to \bZ[\cT_0^\st]$. Note that the both sides of this map are free 
left $\bZ[\Gamma]$-modules of finite rank. 

We define the Steinberg module 
$\St$ as the kernel of the natural augmentation map
\[
\bZ[\bP^1(K)]\to \bZ,
\]
on which the group $\mathit{GL}_2(K)$ acts via
\[
\gamma\circ (x:y)=(x:y)\gamma^{-1},\quad (x:y)\in\bP^1(K). 
\]
We consider it as a left $\bZ[\Gamma]$-module via the natural inclusion
$\Gamma\to \mathit{GL}_2(K)$.
Then the Steinberg module $\St$ is 
a finitely generated projective $\bZ[\Gamma]$-module which sits in the split exact sequence
\begin{equation}\label{EqnExactSt}
\xymatrix{
	0 \ar[r] & \St \ar[r] & \bZ[\bar{\cT}_1^{o,\st}]\ar[r]^{\partial_\Gamma} & \bZ[\cT_0^\st] \ar[r] & 0.
}
\end{equation}
We consider these three left $\bZ[\Gamma]$-modules as right $\bZ[\Gamma]$-modules via the action $[s]\mapsto [s]|_\gamma$.


\subsection{Drinfeld cuspforms and harmonic cocycles}\label{SubsecHC}


For any integer $k\geq 2$ and any $A$-algebra $B$, we denote by $H_{k-2}(B)$ the $B$-submodule of the polynomial ring $B[X,Y]$ consisting of homogeneous 
polynomials 
of degree $k-2$. We consider the left action of the multiplicative monoid $M_2(B)$ on $H_{k-2}(B)$ defined by $(\gamma\circ X,\gamma\circ Y)=(X,Y)\gamma$. 
On $\mathit{GL}_2(B)$, it agrees 
with 
the natural left action on $\Sym^k(\Hom_B(B^2,B))$ induced by the action $\circ$ on $B^2$ after identifying $(X,Y)$ with the dual basis for the basis $
((1,0),
(0,1))$ of $B^2$. Put
\[
V_k(B)=\Hom_B(H_{k-2}(B),B).
\]
We denote the dual basis of the free $B$-module $V_k(B)$ with respect to 
the basis $\{X^iY^{k-2-i}\mid 0\leq i\leq k-2\}$ of $H_{k-2}(B)$ by 
\[
\{(X^iY^{k-2-i})^\vee\mid 0\leq i\leq k-2\}.
\]
We also denote by $\circ$ the natural left action of $\mathit{GL}_2(B)$ on $V_k(B)$ induced by that on $H_{k-2}(B)$. For $\gamma=\begin{pmatrix}
a&b\\c&d
\end{pmatrix}\in \mathit{GL}_2(B)$, $P(X,Y)\in H_{k-2}(B)$ and $\omega\in V_k(B)$, this action is given by
\begin{align*}
(\gamma\circ \omega)(P(X,Y))&=\omega(\gamma^{-1}\circ P(X,Y))\\
&=\det(\gamma)^{2-k}\omega(P(dX-cY,-bX+aY))
\end{align*}
as in \cite[p.~51]{Boeckle}. 
The group $\Gamma$ acts on $H_{k-2}(B)$ and $V_k(B)$ via the natural map $\Gamma\to \mathit{GL}_2(B)$. Moreover, the monoid
\[
M^{-1}=\{\xi\in \mathit{GL}_2(K)\mid \xi^{-1}\in M_2(A)\}
\]
acts on $V_k(B)$ by
\[
(\xi\circ \omega)(P(X,Y))=\omega(\xi^{-1}\circ P(X,Y)).
\]

Put $\cV_k(B)=\St\otimes_{\bZ[\Gamma]}V_k(B)$ and
\[
\cL_{1,k}(B)=\bZ[\bar{\cT}_1^{o,\st}]\otimes_{\bZ[\Gamma]}V_k(B),\quad \cL_{0,k}(B)=\bZ[\cT_0^{\st}]\otimes_{\bZ[\Gamma]}
V_k(B).
\]
We have the split exact sequence
\begin{equation}\label{EqnExactVLL}
\xymatrix{
	0 \ar[r] & \cV_k(B) \ar[r] & \cL_{1,k}(B) \ar[r]^{\partial_\Gamma\otimes 1} &\cL_{0,k}(B)\ar[r] & 0
}
\end{equation}
which is functorial on $B$ and compatible with any base change of $B$. Let $B'$ be any $A$-subalgebra of $B$. Since the $\bZ[\Gamma]$-module $\St$ is 
projective, 
the natural maps $\cV_k(B')\to \cV_k(B)$, $\cL_{1,k}(B')\to \cL_{1,k}(B)$ and $\cL_{0,k}(B')\to \cL_{0,k}(B)$ are injective. 

Let $\Lambda_1\subseteq \cT_1^{o,\st}$ be a complete set of 
representatives of 
$\Gamma\backslash \cT_1^{o,\st}/\{\pm 1\}$. By \cite[Ch.~II, \S1.2, Corollary]{Se}, for any element $e\in \cT_1^{o,\st}$ we can write uniquely
\begin{equation}\label{EqnRe}
r(e)=\vep_e\gamma_e e \quad(\vep_e\in\{\pm 1\}, \gamma_e\in \Gamma, r(e)\in \Lambda_1).
\end{equation} 
Note that $r(e)$, $\vep_e$ and $\gamma_e$ depend on the choice of $\Lambda_1$.
The right $\bZ[\Gamma]$-module $\bZ[\bar{\cT}_1^{o,\st}]$ is free with basis $\{[e]\mid e\in \Lambda_1\}$ and thus, for any $A$-algebra $B
$, any element $x$ of $\cL_{1,k}(B)$ can be written uniquely as
\[
x=\sum_{e\in\Lambda_1} [e]\otimes \omega_e, \quad \omega_e\in V_k(B).
\]

\begin{dfn}\label{DfnHC}
	Let $M$ be a module. A map $c:\cT_1^o\to M$ is said to be a harmonic cocycle if the following conditions are satisfied:
	\begin{enumerate}
		\item\label{DfnHC-h} For any $v\in \cT_0$, we have 
		\[
		\sum_{e\in\cT_1^o,\ t(e)=v}c(e)=0.
		\]
		\item\label{DfnHC-o} For any $e\in \cT_1^o$, we have $c(-e)=-c(e)$.
		\end{enumerate}
	\end{dfn}

Any harmonic cocycle $c$ is determined by its values at $\Gamma$-stable edges, as follows. For any $e\in \cT^{o}_1$, an edge $e'\in \cT^{o,\st}_1$ is said to be 
a 
source of $e$ if the following conditions hold:
\begin{itemize}
	\item When $e$ is $\Gamma$-stable, we require $e'=e$.
	\item When $e$ is $\Gamma$-unstable, we require that a vertex $v$ of $e'$ is $\Gamma$-unstable, $e$ lies on the unique half line from $v$ to $b(v)$ and $e$ 
	has 
	the same orientation as $e'$ with respect to this half line.
\end{itemize}
We denote by $\src(e)$ the set of sources of $e$. Then Definition \ref{DfnHC} (\ref{DfnHC-h}) gives
\begin{equation}\label{EqnSrc}
c(e)=\sum_{e'\in \src(e)}c(e').
\end{equation}
Moreover, for any $\gamma\in \Gamma$, we have 
\begin{equation}\label{EqnGammaSrc}
\src(\gamma(e))=\gamma(\src(e)),\quad \src(-e)=-\src(e).
\end{equation}

For any $A$-algebra $B$, we denote by $C_k^\har(\Gamma, B)$ the set of harmonic cocycles $c:\cT_1^o\to V_k(B)$ which is $\Gamma$-equivariant (namely, $c(\gamma
(e))=\gamma\circ c(e)$ for any $\gamma\in \Gamma$ and $e\in \cT_1^o$). For any rigid analytic function $f$ on $\Omega$ and $e\in \cT^o_1$, we can define
an element $\Res(f)(e)\in V_k(\bC_\infty)$, which gives an isomorphism of $\bC_\infty$-vector spaces
\[
\Res_\Gamma: S_k(\Gamma)\to C_k^\har(\Gamma,\bC_\infty),\quad f\mapsto (e\mapsto \Res(f)(e))
\]
(\cite[Theorem 16]{Tei}, see also \cite[Theorem 5.10]{Boeckle}).
By \cite[(17)]{Boeckle}, the slash operator defined by
\[
(f|_k\gamma)(z)=\det(\gamma)^{k-1}(cz+d)^{-k}f\left(\frac{az+b}{cz+d}\right),\quad \gamma=\begin{pmatrix}
	a&b\\c&d
	\end{pmatrix}\in \mathit{GL}_2(K)
\]
satisfies $\Res(f|_k\gamma)(e)=\gamma^{-1}\circ \Res(f)(\gamma(e))$.


On the other hand, the argument in \cite[p.~506]{Tei} shows that for any $A$-algebra $B$, we have a $B$-linear isomorphism
\[
\Phi_\Gamma: C_k^\har(\Gamma,B)\to \cV_k(B), \quad \Phi_\Gamma(c)=\sum_{e\in\Lambda_1} [e]\otimes c(e),
\]
which is independent of the choice of a complete set of representatives $\Lambda_1$.
This implies that, for any morphism $B\to B'$ of $A$-algebras, the natural map
\[
C_k^\har(\Gamma,B)\otimes_B B'\to C_k^\har(\Gamma,B')
\]
is an isomorphism. Moreover, we obtain an isomorphism
\[
\Phi_\Gamma\circ \Res_\Gamma: S_k(\Gamma)\to \cV_k(\bC_\infty).
\]
In particular, for any $A$-subalgebra $B$ of $\bC_\infty$, we have an injection
\[
\cV_k(B)\to \cV_k(\bC_\infty)\simeq S_k(\Gamma).
\]




\subsection{Hecke operators}\label{SubsecHecke}

For any non-zero element $Q\in A$, we have a Hecke operator $T_Q$ 
acting on $S_k(\Gamma)$ defined as follows. Write
\[
\Gamma \begin{pmatrix}
1&0\\0&Q
\end{pmatrix}
\Gamma=\coprod_{i\in I(\Gamma,Q)} \Gamma\xi_i,
\]
where $\{\xi_i\mid i\in I(\Gamma,Q)\}$ is a complete set of representatives of the right coset space $\Gamma\backslash\Gamma \begin{pmatrix}
1&0\\0&Q
\end{pmatrix}
\Gamma$.
For any $f\in S_k(\Gamma)$, we put
\[
T_Qf=\sum_{i\in I(\Gamma,Q)} f|_k\xi_i.
\]


For any $A$-algebra $B$, we define a Hecke operator $T_Q^\har$ on $C_k^\har(\Gamma,B)$ as follows. 
Note that $\xi_i^{-1}$ is an element of the monoid $M^{-1}$. For any 
$c\in C_k^\har(\Gamma,B)$ and $e\in \cT^{o}_1$, we put
\[
T^\har_Q(c)(e)=\sum_{i\in I(\Gamma,Q)} \xi_i^{-1}\circ c(\xi_i(e)).
\]
Since $c$ is $\Gamma$-equivariant, we see that $T^\har_Q(c)$ is a harmonic cocycle which is independent of the choice of a complete set of 
representatives $\{\xi_i\mid i\in I(\Gamma,Q)\}$. For any $
\delta\in \Gamma$, the set $\{\xi_i\delta\mid i\in I(\Gamma,Q)\}$ is also a complete set of representatives of the same right coset space. This yields 
$T^\har_Q(c)\in C_k^\har(\Gamma,B)$. By \cite[(17)]{Boeckle}, for any $A$-subalgebra $B$ of $\bC_\infty$, the endomorphism $T_Q^\har$ is identified 
with the restriction on $C_k^\har(\Gamma,B)\subseteq C_k^\har(\Gamma,\bC_\infty)$ of the Hecke operator $T_Q$ on $S_k(\Gamma)$ via 
the isomorphism $\Res_{\Gamma}: S_k(\Gamma)\to C_k^\har(\Gamma,\bC_\infty)$.

We also introduce a Hecke operator $T_{1,Q}$ on $\cL_{1,k}(B)$ as follows. We denote by $C_{1,k}^\pm(\Gamma,B)$ the set of $\Gamma$-equivariant 
maps $c:\cT^{o,\st}_1\to V_k(B)$ satisfying $c(-e)=-c(e)$ for any $e\in \cT^{o,\st}_1$. Then the map
\[
\Phi_{1,\Gamma}: C_{1,k}^\pm(\Gamma,B)\to \cL_{1,k}(B),\quad \Phi_{1,\Gamma}(c)=\sum_{e\in \Lambda_1}[e]\otimes c(e)
\]
is independent of the choice of $\Lambda_1$. By the uniqueness of the expression (\ref{EqnRe}), we see that it is an isomorphism.
For any $c\in C_{1,k}^\pm(\Gamma,B)$ and $e\in \cT^{o,\st}_1$, we put
\[
T^\pm_{1,Q}(c)(e)=\sum_{i\in I(\Gamma,Q)}\sum_{e'\in \src(\xi_i(e))} \xi_i^{-1}\circ c(e').
\]
By (\ref{EqnGammaSrc}), it is independent of the choice of $\{\xi_i\}$, and the same argument as in the case of $T^\har_Q$ shows that it defines 
an endomorphism $T^\pm_{1,Q}$ on $C_{1,k}^\pm(\Gamma,B)$.
Now we put
\[
T_{1,Q}=\Phi_{1,\Gamma}\circ T^\pm_{1,Q}\circ \Phi_{1,\Gamma}^{-1}.
\]
From the construction, we see that $T_{1,Q}$ is independent of the choices of $\Lambda_1$ and $\{\xi_i\}$.

For an explicit description of $T_{1,Q}$, fix a complete set of representatives $\Lambda_1$ and 
take any element $x=\sum_{e\in \Lambda_1} [e]\otimes 
\omega_{e}$ of $\cL_{1,k}(B)$. For any $e'\in \cT^{o,\st}_1$, we have
\[
\Phi_{1,\Gamma}^{-1}(x)(e')=\vep_{e'}\gamma_{e'}^{-1}\circ \omega_{r(e')},
\]
where $\vep_{e'}$, $\gamma_{e'}$ and $r(e')$ are defined as (\ref{EqnRe}) using $\Lambda_1$. Hence we obtain
\begin{equation}\label{EqnDefT1}
T_{1,Q}(x)=\sum_{e\in \Lambda_1}[e]\otimes\sum_{i\in I(\Gamma,Q)}\sum_{e'\in \src(\xi_i(e))} 
\vep_{e'}(\xi_i^{-1}\gamma_{e'}^{-1})\circ \omega_{r(e')}.
\end{equation}

\begin{prop}\label{PropInteg}
The restriction of $T_{1,Q}$ on the submodule $\cV_k(B)\subseteq \cL_{1,k}(B)$ agrees with $T_Q^\har$ 
		via the isomorphism $\Phi_{\Gamma}: C_k^\har(\Gamma,B)\to\cV_k(B)$. In particular, $\cV_k(B)$ is stable under $T_{1,Q}$, and 
		if $B$ is an $A$-subalgebra of $\bC_\infty$, then $\cV_k(B)$ defines a $B$-lattice of $S_k(\Gamma)$ which is stable under Hecke 
		operators.
\end{prop}
\begin{proof}
Take any $c\in 
	C_k^\har(\Gamma,B)$. Since $c(r(e'))=\vep_{e'}\gamma_{e'}c(e')$, (\ref{EqnSrc}) yields
	\begin{align*}
	T_{1,Q}(\Phi_{\Gamma}(c))&=\sum_{e\in \Lambda_1}[e]\otimes\sum_{i\in I(\Gamma,Q)}\sum_{e'\in \src(\xi_i(e))} \xi_i^{-1}\circ c(e')\\
	&=\sum_{e\in \Lambda_1}[e]\otimes\sum_{i\in I(\Gamma,Q)} \xi_i^{-1}\circ c(\xi_i(e))=\sum_{e\in \Lambda_1}[e]\otimes T^\har_Q(c)(e),
	\end{align*}
	which agrees with $\Phi_{\Gamma}(T^\har_Q(c))$.
	\end{proof}





\section{Variation of Gouv\^{e}a-Mazur type}\label{SecGM}

Let $\frn\in A$ be a non-zero polynomial which is prime to $\wp$ and $r\geq 1$ an integer. 
For any $A$-algebra 
$B$ and any integer $m\geq 1$, put
\[
B_m=B/\wp^m B.
\]
Note that, since we have the canonical section $[-]:\kappa(\wp)\to \cO_{K_\wp}$ of the natural surjection $\cO_{K_\wp}\to \kappa(\wp)$, 
we can consider $B_m$ canonically as a $\kappa(\wp)$-algebra.

Let $\Theta$ be any subgroup of $1+\wp A_r$. We define
\[
\Gamma_0^\Theta(\wp^r)=\left\{\gamma \in \mathit{SL}_2(A)\ \middle|\ \gamma\bmod \wp^r \in \begin{pmatrix}
\Theta&*\\0&\Theta
\end{pmatrix}\right\}\subseteq \Gamma_1(\wp)
\]
and $\Gamma_1^\Theta(\frn,\wp^r)=\Gamma_1(\frn)\cap \Gamma_0^\Theta(\wp^r)$. 
The subgroup $\Gamma_1^\Theta(\frn,\wp^r)$ of $\mathit{SL}_2(A)$ is $p'$-torsion free 
and contains $\Gamma_1^{\{1\}}(\frn,\wp^r)=\Gamma_1(\frn\wp^r)$. When $\Theta=1+\wp A_r$, we also denote 
$\Gamma_0^\Theta(\wp^r)$ and $\Gamma_1^\Theta(\frn,\wp^r)$ 
by $\Gamma_0^p(\wp^r)$ and $\Gamma_1^p(\frn,\wp^r)$, respectively. 
For $\Gamma_1^\Theta(\frn,\wp^r)$, we fix a complete set of representatives $\Lambda_1$ as in \S\ref{SubsecHC}.

For Hecke operators of level $\Gamma_1^\Theta(\frn,\wp^r)$, we also write
\[
U=T_\wp,\quad U_1=T_{1,\wp}. 
\]
Let $d(k,a)$ be the dimension of the generalized $U$-eigenspace in $S_k(\Gamma_1^\Theta(\frn,\wp^r))$ of slope $a$.
In this section, we prove $p$-adic local constancy results for $d(k,a)$ with respect to $k$,
which generalize the Gouv\^{e}a-Mazur conjecture \cite[Theorem 1.1]{Ha_GVt} for the case of level $\Gamma_1(t)$. 


\subsection{Hecke operators of level $\Gamma_1^\Theta(\frn,\wp^r)$}\label{SubsecDCG}

Let $Q\in A$ be any non-zero element. Write
\[
\Gamma_1^\Theta(\frn,\wp^r)
\begin{pmatrix}
1&0\\0&Q
\end{pmatrix}\Gamma_1^\Theta(\frn,\wp^r)=\coprod_{i\in I(Q)} \Gamma_1^\Theta(\frn,\wp^r) \xi_i.
\]
For any $\gamma\in \Gamma_1^\Theta(\frn,\wp^r)$, $i\in I(Q)$ and $\lambda\in \kappa(\wp)^\times$, we have
\begin{equation}\label{EqnBetaCong}
\gamma\xi_i\equiv \begin{pmatrix}
1&*\\0&Q
\end{pmatrix},\quad \gamma\begin{pmatrix}
\lambda^{-1}&0\\0&\lambda
\end{pmatrix}\equiv \begin{pmatrix}
\lambda^{-1}&*\\0&\lambda
\end{pmatrix} \bmod \wp.
\end{equation}

Consider the Hecke operator $T_Q$ acting on the $\bC_\infty$-vector space $S_k(\Gamma_1^\Theta(\frn,\wp^r))$, which preserves 
the $A$-lattice $\cV_k(A)$ by Proposition \ref{PropInteg}. 
To describe it explicitly for the case where $Q$ is irreducible, we fix a complete set of representatives $R_Q$ of $A/(Q)$. 
When $Q$ divides $\frn\wp^r$, we have $I(Q)=R_Q$ and 
\[
(T_Qf)(z)=\frac{1}{Q}\sum_{\beta\in R_Q} f\left(\frac{z+\beta}{Q}\right).
\]
When $Q$ does not divide $\frn\wp^r$, we can find $R,S\in A$ satisfying $RQ-\frn\wp^r S=1$. Put
\[
\eta_\diamond=\begin{pmatrix}
R&S\\\frn\wp^r &Q
\end{pmatrix},\quad 
\xi_\diamond=\begin{pmatrix}
RQ&S\\\frn\wp^r Q&Q
\end{pmatrix}=\eta_\diamond\begin{pmatrix}
Q&0\\0&1
\end{pmatrix}.
\]
Then we have $I(Q)=\{\diamond\}\sqcup R_Q$ and 
\[
(T_Qf)(z)=Q^{k-1}(\langle Q\rangle_{\frn\wp^r} f)(Qz)+\frac{1}{Q}\sum_{\beta\in R_Q} f\left(\frac{z+\beta}{Q}\right),
\]
where $\langle Q\rangle_{\frn\wp^r}$ is the diamond operator acting on $S_k(\Gamma_1^\Theta(\frn,\wp^r))$ defined by $f\mapsto f|_k\eta_\diamond$.

Note that the natural map 
\[
\mathit{SL}_2(A)\to \mathit{SL}_2(A/(\frn\wp^r))\simeq \mathit{SL}_2(A/(\frn))\times \mathit{SL}_2(A_r)
\]
is surjective.
For any $\lambda\in \kappa(\wp)^\times$, we choose $\eta_\lambda\in 
\mathit{SL}
_2(A)$ 
satisfying
\begin{equation}\label{EqnDefEta}
\eta_\lambda\bmod \frn=I, \quad \eta_\lambda\bmod \wp^r=\begin{pmatrix}
[\lambda]^{-1}&0\\0&[\lambda]
\end{pmatrix}
\end{equation}
and put
\[
\langle \lambda\rangle_{\wp^r}f=f|_k\eta_\lambda.
\]
By 
\begin{equation}\label{EqnGammaEta}
\Gamma_1(\frn\wp^r)\subseteq \Gamma_1^\Theta(\frn,\wp^r),\quad  \eta_\lambda^{-1}\Gamma_1^\Theta(\frn,\wp^r)\eta_\lambda=\Gamma_1^\Theta(\frn,\wp^r),
\end{equation}
this is independent of the choice of $\eta_\lambda$ and defines
an action of $\kappa(\wp)^\times$ on $S_k(\Gamma_1^\Theta(\frn,\wp^r))$. 

For any $\kappa(\wp)[\kappa(\wp)^\times]$-module $M$ and any character $\chi:\kappa(\wp)^\times\to \kappa(\wp)^\times$, 
we denote by $M(\chi)$ the maximal $\kappa(\wp)$-subspace of $M$ on which any $\lambda\in \kappa(\wp)^\times$ acts via $\chi(\lambda)$. 
Since the order of the group $\kappa(\wp)^\times$ is prime to $p$, we have 
the projector
\[
\vep_\chi:M\to M(\chi),\quad \vep_\chi(m)=-\sum_{\lambda\in\kappa(\wp)^\times} \chi(\lambda)^{-1}(\lambda\cdot m)
\]
and the 
decomposition into $\chi$-parts
\[
M=\bigoplus_\chi M(\chi),
\]
where the sum runs over the set of such characters $\kappa(\wp)^\times\to \kappa(\wp)^\times$.

We consider $\bar{K}$ as a $\kappa(\wp)$-algebra by the unique map $\kappa(\wp)\to \bar{K}$ which commutes the diagram
\[
\xymatrix{
	\kappa(\wp)\ar[r]\ar[dr]_{[-]}& \bar{K}\ar[d]^{\iota_\wp}\\
	& \bC_\wp.
}
\]
Then we have
\[
S_k(\Gamma_1^\Theta(\frn,\wp^r))=\bigoplus_{\chi} S_k(\Gamma_1^\Theta(\frn,\wp^r))(\chi).
\]


Note that, when an irreducible polynomial $Q$ does not divide $\frn\wp^r$, we may further assume that $\eta_\lambda$ satisfies
\[
\eta_\lambda=\begin{pmatrix}
a&b\\c&d
\end{pmatrix},\quad a\notin (Q).
\]
Using this, for any irreducible polynomial $Q$ we can show
\[
\Gamma_1^\Theta(\frn,\wp^r)
\eta_\lambda^{-1}\begin{pmatrix}
1&0\\0&Q
\end{pmatrix}\eta_\lambda\Gamma_1^\Theta(\frn,\wp^r)\\
=\Gamma_1^\Theta(\frn,\wp^r)
\begin{pmatrix}
1&0\\0&Q
\end{pmatrix}\Gamma_1^\Theta(\frn,\wp^r).
\]
Then (\ref{EqnGammaEta}) yields
\begin{equation}\label{EqnGammaEtaCoset}
\begin{split}
\coprod_{i\in I(Q)} &\Gamma_1^\Theta(\frn,\wp^r) \xi_i\eta_\lambda=\Gamma_1^\Theta(\frn,\wp^r)
\begin{pmatrix}
1&0\\0&Q
\end{pmatrix}\eta_\lambda\Gamma_1^\Theta(\frn,\wp^r)\\
&=\Gamma_1^\Theta(\frn,\wp^r)
\eta_\lambda
\begin{pmatrix}
1&0\\0&Q
\end{pmatrix}\Gamma_1^\Theta(\frn,\wp^r)=\coprod_{i\in I(Q)} \Gamma_1^\Theta(\frn,\wp^r) \eta_\lambda\xi_i.
\end{split}
\end{equation}
Thus $T_Q$ commutes with $\langle \lambda\rangle_{\wp^r}$ and
$S_k(\Gamma_1^\Theta(\frn,\wp^r))(\chi)$ is stable under Hecke operators. We denote by $d(k,\chi,a)$ be the dimension of the generalized $U$-eigenspace in 
$S_k(\Gamma_1^\Theta(\frn,\wp^r))(\chi)$ of slope $a$. To indicate the level, we often write
\[
d(k,a)=d(\Gamma_1^\Theta(\frn,\wp^r),k,a),\quad d(k,\chi,a)=d(\Gamma_1^\Theta(\frn,\wp^r),k,\chi,a).
\]

For any $A$-algebra $B$, we also have the diamond operator $\langle \lambda\rangle_{\wp^r}$
\[
\langle \lambda\rangle_{\wp^r}\in \End(C_k^\har(\Gamma_1^\Theta(\frn,\wp^r),B)),\quad c\mapsto (e\mapsto \eta_\lambda^{-1}\circ c(\eta_\lambda(e))),
\]
which is compatible with that on $S_k(\Gamma_1^\Theta(\frn,\wp^r))$ when $B=\bC_\infty$.
From (\ref{EqnGammaEta}) we see that $e$ is $\Gamma_1^\Theta(\frn,\wp^r)$-stable if and only of $\eta_\lambda(e)$ is, and thus the corresponding operators on $
\cV_k(B)$ and $\cL_{1,k}(B)$ are given by
\begin{equation}\label{EqnDiaL}
\langle \lambda\rangle_{\wp^r}(\sum_{e\in\Lambda_1}[e]\otimes \omega_e)=\sum_{e\in \Lambda_1}[e]\otimes \vep_{\eta_\lambda(e)}(\eta_\lambda^{-1}\gamma^{-1}
_{\eta_\lambda(e)}
)\circ \omega_{r(\eta_\lambda(e))}.
\end{equation}
When $B$ is also a $\kappa(\wp)$-algebra, we have the decomposition
\[
C_k^\har(\Gamma_1^\Theta(\frn,\wp^r),B)=\bigoplus_\chi C_k^\har(\Gamma_1^\Theta(\frn,\wp^r),B)(\chi)
\]
and similarly for $\cL_{1,k}(B)$ and $\cV_k(B)$. These summands are stable under 
Hecke operators by 
(\ref{EqnGammaEtaCoset}).


\subsection{Weight reduction}\label{SubsecWtRed}

Let $N\geq 1$ be any integer. For any $A$-algebra $B$, the $B$-linear map
\[
\mu_{k,N}: H_{k-2}(B)\to H_{k-2+N}(B),\quad X^iY^{k-2-i}\mapsto X^{i+N}Y^{k-2-i}
\]
induces the dual map
\[
\rho_{k,N}: V_{k+N}(B)\to V_k(B),\quad (X^iY^{k+N-2-i})^\vee\mapsto \begin{cases} 
(X^{i-N}Y^{k+N-2-i})^\vee & (i\geq N) \\ 0 & (i<N)\end{cases}.
\]
It is a surjection whose kernel is
\[
V_{k+N}^{<N}(B)=\bigoplus_{i<N} B (X^iY^{k+N-2-i})^\vee.
\]

\begin{lem}\label{LemWtRedV}
	Let $n\geq 0$ be any non-negative integer, 
	$\bar{B}$ any $A_{p^n}$-algebra and $\lambda\in \kappa(\wp)^\times$. Let $\xi\in M_2(A)$ be any element satisfying
	\[
	\xi=\begin{pmatrix}
	a&b\\c&d
	\end{pmatrix},\quad a\bmod \wp=\lambda,\quad c\equiv 0 \bmod \wp.
	\]
	Let $m$ be the order of $\lambda$ in $\kappa(\wp)^\times$. 
	Then, for any element $\omega\in V_{k+p^n m}(\bar{B})$, we have
	\[
	\xi^{-1}\circ \rho_{k,p^n m}(\omega)=\rho_{k,p^n m}(\xi^{-1}\circ\omega).
	\]
	In particular, for any integer $m'\geq 1$, the map $\rho_{k,p^nm'}:V_{k+p^nm'}(\bar{B})\to V_k(\bar{B})$ is 
	$\Gamma_1^\Theta(\frn,\wp^r)$-equivariant and its kernel $V_{k+p^nm'}^{<p^n m'}(\bar{B})$ is $\Gamma_1^\Theta(\frn,\wp^r)$-stable.
	\end{lem}
\begin{proof}
	In the ring $A_{p^n}$, we can write $a=[\lambda]+\wp a'$ with some $a'\in A_{p^n}$. 
	For any integer $i\in [0,k-2]$, the assumption $\wp^{p^n}\bar{B}=0$
implies
	\begin{align*}
	\xi\circ\mu_{k,p^n m}(X^iY^{k-2-i})&=(aX+cY)^{p^n m+i}(bX+dY)^{k-2-i}\\
	&=(a^{p^n}X^{p^n}+c^{p^n}Y^{p^n})^m (aX+cY)^{i}(bX+dY)^{k-2-i}\\
	&=([\lambda]^{p^n}X^{p^n})^m(aX+cY)^{i}(bX+dY)^{k-2-i}\\
	&=X^{p^n m}(aX+cY)^{i}(bX+dY)^{k-2-i}\\
   &=\mu_{k,p^n m}(\xi\circ (X^iY^{k-2-i})).
	\end{align*}
	Taking the dual yields the lemma.
	\end{proof}

By Lemma \ref{LemWtRedV}, for any $A_{p^n}$-algebra $\bar{B}$ and any integer $m'\geq 1$, we obtain the surjection
\[
1\otimes \rho_{k,p^nm'}:\cV_{k+p^nm'}(\bar{B})\to \cV_k(\bar{B})
\]
and similarly for $\cL_{1,k}(\bar{B})$.

\begin{lem}\label{LemWtRed}
	For any $A_{p^n}$-algebra $\bar{B}$, the maps
	\[
	1\otimes \rho_{k,p^n}:\cV_{k+p^n}(\bar{B})\to \cV_k(\bar{B}),\quad \cL_{1,k+p^n}(\bar{B})\to \cL_{1,k}(\bar{B})
	\]
	commute with Hecke operators. Moreover, the maps 
	\[
	1\otimes \rho_{k,p^n(q^d-1)}:\cV_{k+p^n(q^d-1)}(\bar{B})\to \cV_k(\bar{B}),\quad \cL_{1,k+p^n(q^d-1)}(\bar{B})\to \cL_{1,k}(\bar{B})
	\]
	commute with $\langle\lambda\rangle_{\wp^r}$ for any $\lambda\in\kappa(\wp)^\times$. In particular, the $\bar{B}$-submodules
	\[
	\cV_{k+p^n}^{<p^n}(\bar{B}),\quad \cV_{k+p^n(q^d-1)}^{<p^n(q^d-1)}(\bar{B})
	\]
	are stable under Hecke operators.
	\end{lem}
\begin{proof}
	It is enough to show the assertions on $\cL_{1,k}(\bar{B})$. 
	By (\ref{EqnDefT1}) and (\ref{EqnDiaL}), we reduce ourselves to 
	showing that, for any $\gamma\in \Gamma_1^\Theta(\frn,\wp^r)$, $i\in I(Q)$, $\lambda\in\kappa(\wp)^\times$, $\omega \in V_{k+p^n}(\bar{B})$ and $\omega' \in 
	V_{k
	+p^n(q^d-1)}(\bar{B})$, we have
	\begin{align*}
	(\gamma\xi_i)^{-1}\circ \rho_{k,p^n}(\omega)&=\rho_{k,p^n}((\gamma\xi_i)^{-1}\circ \omega),\\ (\gamma\eta_\lambda)^{-1}\circ \rho_{k,p^n(q^d-1)}(\omega')&=
	\rho_{k,p^n(q^d-1)}
	((\gamma\eta_\lambda)^{-1}\circ \omega').
	\end{align*}
	By (\ref{EqnBetaCong}), this follows from Lemma \ref{LemWtRedV}.
\end{proof}


\subsection{Dimension of slope zero cuspforms}\label{SubsecDimOrd}

Using harmonic cocycles, the proofs of \cite[Corollary 8.2 and Proposition 8.3]{Hida_TotReal} can be adapted to obtain constancy 
results for the dimension of slope zero cuspforms with respect to the weight and the level at $\wp$. First we prove the following key lemma.


\begin{lem}\label{LemPreGliss}
	Let $B$ be any flat $A$-algebra. For any $s\in \St$ and any integer $j\in [0,k-2]$, the element $s\otimes (X^jY^{k-2-j})^\vee\in \cV_k(B)$ 
	satisfies
	\[
	U(s\otimes (X^jY^{k-2-j})^\vee)\in \wp^{k-2-j}\cV_k(B).
	\]  
	\end{lem}
\begin{proof}
	For any non-negative integer $m$, we have the commutative diagram with exact rows
	\[
	\xymatrix{
		0 \ar[r] & \cV_k(B) \ar[r]\ar[d] & \cL_{1,k}(B) \ar[r]^{\partial_\Gamma\otimes 1}\ar[d] &\cL_{0,k}(B)\ar[r]\ar[d] & 0\\
		0 \ar[r] & \cV_k(B_m) \ar[r] & \cL_{1,k}(B_m) \ar[r]^{\partial_\Gamma\otimes 1} &\cL_{0,k}(B_m)\ar[r] & 0.
	}
	\]
	Since the structure map $A\to B$ is flat, we see that $\wp^{m}\cV_k(B)$ and $\wp^m\cL_{1,k}(B)$ are the kernels of the left two vertical maps. Thus it 
	suffices to show 
	$U_1(s\otimes (X^jY^{k-2-j})^\vee)\in \wp^{k-2-j}\cL_{1,k}(B)$. 
	
	Any element of $\St$ is a $\bZ$-linear combination of elements of $\bZ[\bar{\cT}^{o,\st}_1]$ of the form $[e]|_\alpha$ 
	with  $e\in \Lambda_1$ and $\alpha\in \Gamma_1^\Theta(\frn,\wp^r)$. Moreover, for any $\omega\in V_k(B)$, we have 
	$[e]|_\alpha\otimes\omega=[e]\otimes\alpha\circ \omega$.
	By (\ref{EqnDefT1}), it is enough to show that, 
	for any $i\in I(\wp)$, $\gamma\in \Gamma_1^\Theta(\frn,\wp^r)$ and integers $j,l\in [0,k-2]$, we have
	\[
	((\gamma\xi_i)^{-1}\circ (X^jY^{k-2-j})^\vee)(X^lY^{k-2-l})\in \wp^{k-2-j}B.
	\]
	
	Write $\gamma\xi_i=\begin{pmatrix}
	a&b\\c&d
	\end{pmatrix}$. 
	Then the above evaluation is equal to
	\[
	(X^jY^{k-2-j})^\vee((aX+cY)^l(bX+dY)^{k-2-l}).
	\]
	By (\ref{EqnBetaCong}) we have $c,d\equiv 0 \bmod \wp$ and the coefficient of $X^jY^{k-2-j}$ in the product 
	$(aX+cY)^l(bX+dY)^{k-2-l}$ is divisible by $\wp^{k-2-j}$. This concludes the proof.
\end{proof}

\begin{prop}\label{PropOrdAgree}
	\begin{enumerate}
		\item\label{PropOrdAgree-1} $d(\Gamma_1^\Theta(\frn,\wp^r),k,0)$ is independent of $k$.
		\item\label{PropOrdAgree-chi} For any character $\chi:\kappa(\wp)^\times\to \kappa(\wp)^\times$, we have
		\[
		k_1\equiv k_2 \bmod q^d-1\Rightarrow d(\Gamma_1^\Theta(\frn,\wp^r),k_1,\chi,0)=d(\Gamma_1^\Theta(\frn,\wp^r),k_2,\chi,0).
		\]
	\end{enumerate}
\end{prop}
\begin{proof}
	Note that $d(\Gamma_1^\Theta(\frn,\wp^r),k,0)$ is equal to the degree of the polynomial
	\[
	\det(I-UX;\cV_k(\kappa(\wp))).
	\]
	By Lemma \ref{LemWtRed} for $n=0$, we have the exact sequence
	\[
	\xymatrix{
		0\ar[r]& \cV_{k+1}^{<1}(\kappa(\wp))\ar[r] & \cV_{k+1}(\kappa(\wp))\ar[r]& \cV_k(\kappa(\wp))\ar[r] & 0
	}
	\]
	whose maps are compatible with Hecke operators. Since $(k+1)-2>0$, Lemma \ref{LemPreGliss} implies $U=0$ on $\cV_{k+1}^{<1}(\kappa(\wp))$ 
	and
	thus we have
	\[
	\det(I-UX;\cV_{k+1}^{<1}(\kappa(\wp)))=1,
	\]
	which yields the assertion (\ref{PropOrdAgree-1}). Since Lemma \ref{LemWtRed} also gives the exact sequence
	\[
	\xymatrix{
		0\ar[r]& \cV_{k+p^d-1}^{<p^d-1}(\kappa(\wp))(\chi)\ar[r] & \cV_{k+p^d-1}(\kappa(\wp))(\chi)\ar[r]& \cV_k(\kappa(\wp))(\chi)\ar[r] & 0,
	}
	\]
	the assertion (\ref{PropOrdAgree-chi}) follows similarly.
\end{proof}

\begin{prop}\label{PropHida}
	$d(\Gamma_1^p(\frn,\wp^r),k,0)$ and $d(\Gamma_1^p(\frn,\wp^r),k,\chi,0)$ are independent of $r\geq 1$.
\end{prop}
\begin{proof}
	Put $\Gamma_r=\Gamma_1^p(\frn,\wp^r)$. Let $\bar{\kappa}$ be an algebraic closure of $\kappa(\wp)$.
	We reduce ourselves to showing that the multiplicities of non-zero eigenvalues of $U$ acting on
	$C^\har_k(\Gamma_r,\bar{\kappa})$ and $C^\har_k(\Gamma_r,\bar{\kappa})(\chi)$ are independent of $r$. 
	These are the same as the dimensions of the generalized
	eigenspaces 
	\[
	C^\har_k(\Gamma_r,\bar{\kappa})^\ord,\quad C^\har_k(\Gamma_r,\bar{\kappa})(\chi)^\ord
	\] 
	of non-zero eigenvalues, respectively.
	
	Since any $c\in C_k^\har(\Gamma_r,\bar{\kappa})$ is also $\Gamma_{r+1}$-equivariant, we have the natural inclusion
	\[
	\iota: C^\har_k(\Gamma_r,\bar{\kappa})\to C^\har_k(\Gamma_{r+1},\bar{\kappa}).
	\]
	Since we have
	\[
	\Gamma_{r+1}\begin{pmatrix}
	1&0\\0&\wp
	\end{pmatrix}\Gamma_r=\coprod_{\beta\in R_\wp} \Gamma_{r+1}\xi_\beta,\quad \xi_\beta=\begin{pmatrix}
	1&\beta\\0&\wp
	\end{pmatrix},
	\]
	we obtain a map $s:C^\har_k(\Gamma_{r+1},\bar{\kappa})\to C^\har_k(\Gamma_{r},\bar{\kappa})$ by
	\[
	s(c)(e)=\sum_{\beta\in R_\wp}\xi_\beta^{-1}\circ c(\xi_\beta(e)),
	\]
	which makes the following diagram commutative.
	\[
	\xymatrix{
		C_k^\har(\Gamma_r,\bar{\kappa})\ar[r]^{\iota}\ar[d]_U & C^\har_k(\Gamma_{r+1},\bar{\kappa})\ar[d]^U\ar[ld]_s\\
		C_k^\har(\Gamma_r,\bar{\kappa})\ar[r]_{\iota}& C^\har_k(\Gamma_{r+1},\bar{\kappa})
	}
	\]
	From this we see that $\iota$ and $s$ commute with $U$ and, since $U$ is isomorphic on $C^\har_k(\Gamma_r,\bar{\kappa})^\ord$, the map $\iota$ gives 
	an isomorphism
	\[
	\iota^\ord: C^\har_k(\Gamma_r,\bar{\kappa})^\ord\to C^\har_k(\Gamma_{r+1},\bar{\kappa})^\ord.
	\]
	This settles the assertion on $d(\Gamma_1^p(\frn,\wp^r),k,0)$. Moreover, since the diamond operator $\langle\lambda\rangle_{\wp^r}$ is independent of the 
	choice of $\eta_\lambda$ satisfying (\ref{EqnDefEta}), we also have 
	\[
	\langle\lambda\rangle_{\wp^{r+1}}\circ \iota=\iota\circ \langle\lambda\rangle_{\wp^r}.
	\]
	Since $U$ commutes with diamond operators, the map $\iota^\ord$ also induces an isomorphism
	\[
	C^\har_k(\Gamma_r,\bar{\kappa})(\chi)^\ord\to C^\har_k(\Gamma_{r+1},\bar{\kappa})(\chi)^\ord,
	\]
	from which the assertion on $d(\Gamma_1^p(\frn,\wp^r),k,\chi,0)$ follows.
\end{proof}


\subsection{Representing matrix of $U$}\label{SubsecUop}

Let $E/K_\wp$ be a finite extension of complete 
valuation fields. We extend the normalized $\wp$-adic valuation $v_\wp$ naturally to $E$. We denote by $\cO_E$ the integer ring of $E$.

\begin{lem}\label{LemVFree}
	Suppose that $\frn\wp$ has a prime factor $\pi$ of degree one. Then the right $\bZ[\Gamma_1^\Theta(\frn,\wp^r)]$-module $\St$ 
	is free of rank $[\Gamma_1(\pi):\Gamma_1^\Theta(\frn,\wp^r)]$, where the rank is independent of the choice of such $\pi$.
	\end{lem}
	\begin{proof}
		Note that, from $\Gamma_1^\Theta(\frn,\wp^r)\subseteq \Gamma_1(\frn\wp)$, 
		we see that the former is a subgroup of $\Gamma_1(\pi)$.
		We can show that a fundamental domain of $\Gamma_1(\pi)\backslash \cT$ is the same as the picture of 
		\cite[\S7]{LM}, and that it has no $\Gamma_1(\pi)$-stable vertex and only one $\Gamma_1(\pi)$-stable (unoriented) edge. 
		By (\ref{EqnExactSt}), the right $\bZ[\Gamma_1(\pi)]$-module $\St$ is free of rank one. Thus the right $\bZ[\Gamma_1^\Theta(\frn,\wp^r)]$-module 
		$\St$ is free of rank $[\Gamma_1(\pi):\Gamma_1^\Theta(\frn,\wp^r)]$. Since we have
		\[
		[\Gamma_1(\pi):\Gamma_1^\Theta(\frn,\wp^r)]=[\mathit{SL}_2(A):\Gamma_1^\Theta(\frn,\wp^r)]\left[\mathit{SL}_2(\bF_q):\left\{\begin{pmatrix}
		1&*\\0&1
		\end{pmatrix}\right\}\right]^{-1},
		\]
		the rank is independent of $\pi$.
	\end{proof}

In the sequel, we assume that $\frn\wp$ has a prime factor $\pi$ of degree one. Under this assumption,
Lemma \ref{LemVFree} implies that the right $\bZ[\Gamma_1^\Theta(\frn,\wp^r)]$-module $\St$ is free 
of rank $d$, where we put
\[
d=[\Gamma_1(\pi):\Gamma_1^\Theta(\frn,\wp^r)].
\]
Hence, for any $A$-algebra $B$, the $B$-module $\cV_k(B)$ is free of rank $d(k-1)$.
We fix an ordered basis $\frB_k$ of the free $A$-module $\cV_k(A)$, as follows.
Take an ordered basis $(s_1,\ldots,s_d)$ of the right $\bZ[\Gamma_1^\Theta(\frn,\wp^r)]$-module $\St$.
The set 
\[
\frB_k=\{ v_{i,j}=s_i\otimes (X^jY^{k-2-j})^\vee\mid 1\leq i\leq d,\ 0\leq j\leq k-2\} 
\]
forms a basis of the $A$-module $\cV_{k}(A)$, and we order it as
\[
v_{1,0},v_{2,0},\ldots,v_{d,0},v_{1,1},v_{2,1},\ldots,v_{d,1},v_{1,2},\ldots.
\]
For any $A$-algebra $B$, the ordered basis of the $B$-module $\cV_k(B)$
induced by $\frB_k$ is also denoted abusively by $\frB_k$.
We denote by $U^{(k)}$ the representing matrix of $U$ acting on the $\cO_E$-module $\cV_{k}(\cO_E)$ with respect 
to 
the ordered basis $\frB_k$. Then Lemma \ref{LemPreGliss} gives
\begin{equation}\label{EqnGliss}
U(v_{i,j})\in \wp^{k-2-j}\cV_k(\cO_E).
\end{equation}

In order to study perturbation of $U^{(k)}$, we use the following lemma of \cite{Ked}. Note that the assumption $B\in \mathit{GL}_n(F)$ there is superfluous. 
\begin{lem}[\cite{Ked}, Proposition 4.4]\label{LemKedElDiv}
	Let $L$ be any positive integer and $A,B\in M_L(\cO_E)$. Let $s_1\leq s_2\leq\cdots \leq s_L$ be the elementary divisors of $A$. 
	Namely, they are the normalized $\wp$-adic 
	valuations of diagonal entries of the Smith normal form of $A$. Let $s'_1\leq s'_2\leq\cdots \leq s'_L$ be the elementary divisors of $AB$. Then we have
	\[
	s'_i\geq s_i \quad\text{for any }i.
	\]
	The same inequality also holds for the elementary divisors of $BA$.
	\end{lem}

\begin{cor}\label{CorElDiv}
	Suppose that $\frn\wp$ has a prime factor $\pi$ of degree one. Put $d=[\Gamma_1(\pi):\Gamma_1^\Theta(\frn,\wp^r)]$. Let $s_1\leq s_2\leq\cdots \leq 
	s_{d(k-1)}$ 
	be the elementary divisors of $U^{(k)}$. Then we have
	\[
	s_i\geq \left\lfloor\frac{i-1}{d}\right\rfloor.
	\]
	\end{cor}
\begin{proof}
	By (\ref{EqnGliss}), the matrix $U^{(k)}$ can be written as 
	\[
	U^{(k)}=B\diag(\wp^{k-2},\ldots,\wp^{k-2},\ldots,\wp,\ldots,\wp,1,\ldots,1),
	\]
	where $B\in M_{d(k-1)}(\cO_E)$ and the diagonal entries of the last matrix are $\{\wp^j\mid 0\leq j \leq k-2\}$, each with multiplicity $d$. Then the 
	corollary 
	follows from Lemma \ref{LemKedElDiv}.
\end{proof}

\begin{cor}\label{CorWindow}
	Let $n\geq 0$ be any non-negative integer. Then, for some matrices $B_1,B_2,B_3,B_4$ with entries in $\cO_E$, we have
	\[
	U^{(k+p^n)}= \left(\begin{array}{c|c}
	\wp^{k-1}B_1 & B_2\\\hline
	\wp^{p^n}B_3 & U^{(k)}+\wp^{p^n}B_4
	\end{array}\right).
	\]
	\end{cor}
\begin{proof}
	By Lemma \ref{LemWtRed}, the lower right block is congruent to $U^{(k)}$ and the lower left block is zero modulo $\wp^{p^n}$. By (\ref{EqnGliss}), the 
	entries on the upper left block are divisible by $\wp^{k-1}$. This concludes the proof.
\end{proof}

For the $U$-operator acting on $\cV_k(\cO_E)(\chi)$, we have a similar description of its representing matrix $U^{(k)}_\chi$ as follows.

\begin{prop}\label{PropWindow0}
	Suppose that $\frn\wp$ has a prime factor $\pi$ of degree one. Put $d=[\Gamma_1(\pi):\Gamma_1^\Theta(\frn,\wp^r)]$.
	\begin{enumerate}
	\item\label{PropWindow0-ElDiv} For any integer $i\geq 0$, the $i$-th smallest elementary divisor $s_{\chi,i}$ of $U^{(k)}_\chi$ 
	satisfies 
	\[
	s_{\chi,i}\geq \left\lfloor\frac{i-1}{d}\right\rfloor.
	\]
	\item\label{PropWindow0-Window} 
	Let $n\geq 0$ be any non-negative integer. Then, with some bases of $\cV_k(\cO_E)(\chi)$ and $\cV_{k+p^n(q^d-1)}(\cO_E)(\chi)$, the representing matrices
	$U^{(k)}_\chi$ and $U^{(k+p^n(q^d-1))}_\chi$ of $U$ acting on them satisfies
	\[
	U^{(k+p^n(q^d-1))}_\chi= \left(\begin{array}{c|c}
	\wp^{k-1}B_1 & B_2\\\hline
	\wp^{p^n}B_3 & U^{(k)}_\chi+\wp^{p^n}B_4
	\end{array}\right)
	\]
	for some matrices $B_1,B_2,B_3,B_4$ with entries in $\cO_E$. 
	\end{enumerate}
\end{prop}
\begin{proof}
	We have the decomposition 
	\[
	\cV_k(\cO_E)=\bigoplus_\chi \cV_k(\cO_E)(\chi),
	\]
	where each summand is stable under Hecke operators. Thus any elementary divisor of $U^{(k)}_\chi$ is also an elementary divisor of $U^{(k)}$, and 
	$s_{\chi,i}$ equals the $i'$-th smallest elementary divisor $s_{i'}$ of $U^{(k)}$ with some $i'\geq i$. Hence the 
	assertion (\ref{PropWindow0-ElDiv}) follows from Corollary \ref{CorElDiv}.
	
	For (\ref{PropWindow0-Window}), put $m=q^d-1$, $k'=k+p^n m$ and consider the weight reduction map
	\[
	\rho=1\otimes \rho_{k,p^nm}:\cV_{k'}(\cO_{E,p^n})\to \cV_k(\cO_{E,p^n}).
	\]
	By Lemma \ref{LemWtRedV}, we can define the tensor product over $\bZ[\Gamma_1^\Theta(\frn,\wp^r)]$
	\[
	\cV_{k'}^{<p^n m}(\cO_{E,p^n})=\St\otimes_{\bZ[\Gamma_1^\Theta(\frn,\wp^r)]}V_{k'}^{<p^n m}(\cO_{E,p^n}),
	\]
	which sits in the split exact sequence of $\cO_{E,p^n}$-modules
	\[
	\xymatrix{
		0\ar[r] & \cV_{k'}^{<p^n m}(\cO_{E,p^n})\ar[r] & \cV_{k'}(\cO_{E,p^n})\ar[r]^-{\rho} &\cV_k(\cO_{E,p^n})\ar[r]& 0.
	}
	\]
	By Lemma \ref{LemWtRed}, the map $\rho$ is compatible with Hecke operators and $\langle \lambda\rangle_{\wp^r}$ for any $\lambda\in\kappa(\wp)^\times$.
	Thus the map $\rho$ also induces the split exact sequence
	\[
	\xymatrix{
		0\ar[r] & \cV_{k'}^{<p^n m}(\cO_{E,p^n})(\chi) \ar[r] & \cV_{k'}(\cO_{E,p^n})(\chi)\ar[r]^-{\rho} &\cV_k(\cO_{E,p^n})(\chi)\ar[r]& 0.
	}
	\]
	
	Let $\vep_\chi: \cV_{k'}(\cO_E)\to \cV_{k'}(\cO_E)(\chi)$ be the projector to the $\chi$-part. Let $\kappa_E$ be the residue field of $E$. 
	Consider the basis $v_{i,j}=s_i\otimes (X^jY^{k'-2-j})^\vee$ of $\cV_{k'}(\cO_E)$ as before and its image $\bar{v}_{i,j}$ 
	in $\cV_{k'}(\kappa_E)$. Note that, for any $j<p^n m$, the image of $\vep_\chi(v_{i,j})$ in $\cV_{k'}(\cO_{E,p^n})(\chi)$ lies in $\cV_{k'}^{<p^n m}
	(\cO_{E,p^n})(\chi)$.
	Since the set 
	\[
	\{\vep_\chi(\bar{v}_{i,j}) \mid 1\leq i\leq d,\ 0\leq j\leq p^n m-1\}
	\]
	spans the $\kappa_E$-vector space $\cV_{k'}^{<p^n m}(\kappa_E)(\chi)$, there exists a subset $\Sigma\subseteq [1,d]\times [0,p^n m-1]$ such that the 
	elements
	$\vep_\chi(\bar{v}_{i,j})$ for $(i,j)\in \Sigma$ form its basis. 
	
	Now take a lift $\frB_{k',\chi,k}$ of a basis of $\cV_k(\cO_{E,p^n})(\chi)$ by the composite 
	\[
	\cV_{k'}(\cO_{E})(\chi)\to \cV_{k'}(\cO_{E,p^n})(\chi)\overset{\rho}{\to}\cV_k(\cO_{E,p^n})(\chi).
	\]
	Since the image of the set
	\[
	\frB_{k',\chi}=\{\vep_\chi(v_{i,j})\mid (i,j)\in\Sigma\}\sqcup \frB_{k',\chi,k}
	\]
	in $\cV_{k'}(\kappa_{E})(\chi)$ forms its basis, we see that $\frB_{k',\chi}$ itself forms a basis of $\cV_{k'}(\cO_{E})(\chi)$. 
	Moreover, by Nakayama's lemma, the images of $\vep_\chi(v_{i,j})$ in $\cV_{k'}(\cO_{E,p^n})$ for $(i,j)\in \Sigma$ form a basis
	of $\cV_{k'}^{<p^n m}(\cO_{E,p^n})(\chi)$.

	Representing $U$ by the basis $\frB_{k',\chi}$, we see that the lower blocks of the resulting matrix are as stated in 
	(\ref{PropWindow0-Window}). 
	Moreover, since $U$ and $\langle \lambda\rangle_{\wp^r}$ commute with each other, (\ref{EqnGliss}) yields
	\[
	U(\vep_\chi(v_{i,j}))=\vep_\chi(U(v_{i,j}))\in \wp^{k'-2-j}\cV_{k'}(\cO_E)(\chi)
	\]
	for any $j<p^n m$, and thus the upper left block is divisible by $\wp^{k-1}$. This concludes the proof.
\end{proof}


\subsection{Perturbation}\label{SubsecPerturb}

Let $E/K_\wp$ be a finite extension inside $\bC_\wp$. 
Let $V$ be an $E$-vector space of finite dimension and $T:V\to V$ an $E$-linear endomorphism. For an eigenvector of $T$ with eigenvalue $\lambda\in \bC_\wp$, 
we refer to $v_\wp(\lambda)$ as its slope. For any rational number $a$, we denote by $d(T,a)$ the multiplicity of $T$-eigenvalues of slope $a$. If 
$B$ is the representing matrix of $T$ with some basis of $V$, we also denote it by $d(B,a)$.

\begin{prop}\label{PropPerturb}
	Let $d_0$, $n$ and $L$ be positive integers. Let $B\in M_{L}(\cO_E)$ be a matrix such that its $i$-th smallest elementary divisor $s_i$ satisfies $s_i\geq 
	\lfloor
	\frac{i-1}{d_0} \rfloor$ for any $i$. Put $\vep_0=d(B,0)$ and
	\[
	C_1(n,d_0,\vep_0)=p^n\left(\frac{4+d_0 p^n-d_0}{4+2d_0p^n-2\vep_0}\right)\in (0,p^n).
	\]
	Moreover, we put $q_1=r_1=0$ and for any $l\geq 2$, we write $q_l=\lfloor\frac{l-2}{d_0}\rfloor$ and $r_l=l-2-d_0q_l$. We define $C_2(n,d_0,\vep_0)$ as 
	\[
	\min\left\{\frac{2p^n+d_0q_l(q_l-1)+2q_l(r_l+1)}{2(l-\vep_0)}\ \middle|\ \vep_0< l\leq 1+d_0 p^n\right\}
	\]
	and put 
	\[
	C(n,d_0,\vep_0)=\min\{C_1(n,d_0,\vep_0),C_2(n,d_0,\vep_0)\}\in (0,p^n).
	\]
	Let $B'\in M_L(\cO_E)$ be any matrix satisfying $B'-B\in \wp^{p^n}M_L(\cO_E)$. 
	Let $a$ be any non-negative rational number satisfying 
	\[
	a<C(n,d_0,\vep_0).
	\]
	Then we have
	\[
	d(B,a)=d(B',a).
	\]
\end{prop}
\begin{proof}
	We put
	\[
	P_B(X)=\det(I-BX)=\sum b_l X^l,\quad P_{B'}(X)=\det(I-B'X)=\sum b'_l X^l.
	\]
	Then $b_l$ is, up to a sign, the sum of principal $l\times l$ minors of $B$. Since $P_B\equiv P_{B'}\bmod \wp$, we have $d(B',0)=d(B,0)=\vep_0$. 
	From the assumption on elementary divisors, we see that if $i>d_0$, then any $i\times i$ minor of $B$ is divisible by $\wp$. This yields $\vep_0\leq d_0$.

	By \cite[Theorem 4.4.2]{Ked}, for any $l\geq 0$ we have
	\[
	v_\wp(b_l-b'_l)\geq p^n+\sum_{j=1}^{l-1}\min\left\{\left\lfloor
	\frac{j-1}{d_0} \right\rfloor,p^n\right\}.
	\]
	Here we mean that the second term of the right-hand side is zero for $l\leq 1$.
	Let $R$ be the right-hand side of the inequality. We claim that for any $l> \vep_0$, we have
	\[
	 a< C(n,d_0,\vep_0)\Rightarrow R> a(l-\vep_0).
	\]
	Indeed, when $l>1+d_0p^n$, we have
	\begin{align*}
	R&=p^n+\sum_{j=1}^{d_0p^n}\left\lfloor\frac{j-1}{d_0} \right\rfloor+\sum_{j=1+d_0p^n}^{l-1} p^n=p^n(l-d_0p^n)+\frac{1}{2}d_0p^n(p^n-1)\\
	&=\frac{1}{2}p^n(2l-d_0-d_0p^n).
	\end{align*}
	Then $R>a(l-\vep_0)$ if and only if
	\begin{equation}\label{EqnRLarge}
	(p^n-a)l-\frac{1}{2}p^nd_0(1+p^n)+a\vep_0>0.
	\end{equation}
	Since the condition $a<C(n,d_0,\vep_0)$ yields $p^n>a$, the left-hand side of (\ref{EqnRLarge}) is increasing with respect to $l$. 
	Thus (\ref{EqnRLarge}) holds for any $l>1+d_0p^n$ 
	if and only if it holds for 
	$l=2+d_0p^n$, which is equivalent to $a<C_1(n,d_0,\vep_0)$.
	
	On the other hand, when $l\leq 1+d_0p^n$, we have
	\begin{equation}\label{EqnRSmall}
	R=p^n+\frac{1}{2}d_0q_l(q_l-1)+q_l(r_l+1),
	\end{equation}
	from which the claim follows.

	Let $N_B$ and $N_{B'}$ be the Newton polygons of $P_B$ and $P_{B'}$, respectively. It suffices to show that the segments of $N_B$ and $N_{B'}$ with slope 
	less than $C(n,d_0,\vep_0)$ agree with each other. 
	Suppose the contrary and take the smallest slope $a<C(n,d_0,\vep_0)$ satisfying $d(B,a)\neq d(B',a)$.
	
	Let $(l,y)$ be the right endpoint of the segment of slope $a$ in either of $N_B$ or $N_{B'}$. 
	Since $d(B,0)=d(B',0)$, we have $a>0$ and $l>\vep_0$. Then the above claim yields
	\[
	y\leq a (l-\vep_0)<v_\wp(b_l-b'_l).
	\]
	Since $y\in \{v_\wp(b_l),v_\wp(b'_l)\}$, we have $v_\wp(b_l)=v_\wp(b'_l)$. Since $a$ is minimal, this implies that slope $a$ appears in both of $N_B$ and 
	$N_{B'}$. Applying the same argument to the right endpoint of the segment of slope $a$ in the other Newton polygon, we obtain $d(B,a)=d(B',a)$. This is the 
	contradiction.
\end{proof}

By a similar argument, we can show a slightly different perturbation result as follows. 

\begin{prop}\label{PropPerturbGM}
	With the notation in Proposition \ref{PropPerturb}, we suppose that the following 
	conditions hold.
	\begin{enumerate}
		\item\label{PropPerturbGM-large} If $p=2$, then $n\geq 3$ or $d_0-\vep_0\leq 1$.
		\item\label{PropPerturbGM-small} $2p^n> n(d_0n+2+d_0-2\vep_0)$.
	\end{enumerate}
	Then, for any non-negative rational number $a\leq n$, we have
	\[
	d(B,a)=d(B',a).
	\]
	\end{prop}
\begin{proof}
	Let $R$ be as in the proof of Proposition \ref{PropPerturb}. We claim $R> n(l-\vep_0)$ for any $l>\vep_0$ under the assumptions (\ref{PropPerturbGM-large}) and  
	(\ref{PropPerturbGM-small}).
	
	Indeed, when $l>1+d_0p^n$, we have $R>n(l-\vep_0)$ for any such $l$ if and only if $n<C_1(n,d_0,\vep_0)$, namely
	\[
	d_0p^n\left(\frac{1}{2}p^n-n\right)+2(p^n-n)+n\vep_0>\frac{1}{2}d_0p^n.
	\]
	If $p\geq 3$ or $n\geq 3$, then we have $\frac{1}{2}p^n-n\geq \frac{1}{2}$ and the above inequality holds. If $p=2$ and $n<3$, 
it is equivalent to $d_0-\vep_0\leq 1$. Thus, under the condition (\ref{PropPerturbGM-large}), we have $R> n(l-\vep_0)$ in this case.
	
	Let us consider the case of $l\leq 1+d_0p^n$. Note that $l=1$ is allowed only if $\vep_0=0$, in which case the claim holds by $R=p^n>n$.
	For $l\geq 2$, by (\ref{EqnRSmall}) we have $R>n(l-\vep_0)$ if and only if
	\begin{align*}
	2p^n+d_0\left(q_l-n+\frac{r_l+1}{d_0}-\frac{1}{2}\right)^2-d_0\left(-n+\frac{r_l+1}{d_0}-\frac{1}{2}\right)^2>2n(r_l+2-\vep_0).
	\end{align*}
	Note $\frac{r_l+1}{d_0}-\frac{1}{2}\in [-\frac{1}{2},\frac{1}{2}]$. Since $q_l$ and $n$ are integers, we have
	\[
	d_0\left(q_l-n+\frac{r_l+1}{d_0}-\frac{1}{2}\right)^2\geq d_0\left(\frac{r_l+1}{d_0}-\frac{1}{2}\right)^2.
	\]
	Thus the above inequality holds if
	\[
	2p^n+d_0\left(\frac{r_l+1}{d_0}-\frac{1}{2}\right)^2-d_0\left(-n+\frac{r_l+1}{d_0}-\frac{1}{2}\right)^2>2n(r_l+2-\vep_0),
	\]
	which is equivalent to the condition (\ref{PropPerturbGM-small}) and the claim follows. Now the same reasoning as in the proof of Proposition 
	\ref{PropPerturb} shows $d(B,a)=d(B',a)$.
\end{proof}


\subsection{Dimension variation}\label{SubsecGVP}

For the $U$-operators acting on $\cV_k(K_\wp)$ and $\cV(K_\wp)(\chi)$, we denote $d(U,a)$ also by 
\[
d(k,a)=d(\Gamma_1^\Theta(\frn,\wp^r),k,a),\quad d(k,\chi,a)=d(\Gamma_1^\Theta(\frn,\wp^r),k,\chi,a),
\]
respectively. 
Note that they agree with the previously defined ones
for $S_k(\Gamma_1^\Theta(\frn,\wp^r))$ and $S_k(\Gamma_1^\Theta(\frn,\wp^r))(\chi)$.

Now the following theorems give generalizations of \cite[Theorem 1.1]{Ha_GVt}.

\begin{thm}\label{ThmDMFP}
	Suppose that $\frn\wp$ has a prime factor $\pi$ of degree one. Let $n\geq 1$ and $k\geq 2$ be any integers. 
	Put $d=[\Gamma_1(\pi):\Gamma_1^\Theta(\frn,\wp^r)]$ and $\vep=d(k,0)$. 
	Let $a$ be any non-negative rational number satisfying 
	\[
	a<\min\{C(n,d,\vep),k-1\}.
	\]
	Then, for any integer $k'\geq k$, we have
	\[
	k'\equiv k \bmod p^n \Rightarrow d(k',a)=d(k,a).
	\]
\end{thm}
\begin{proof}
	By Proposition \ref{PropOrdAgree} (\ref{PropOrdAgree-1}), we may assume $k'=k+p^n$.
	By Corollary \ref{CorWindow}, we can write $U^{(k+p^n)}+\wp^{p^n}W=V$ with $W\in M_{d(k+p^n-1)}(\cO_{K_\wp})$ and
	\[
	V=\left(\begin{array}{c|c}
	\wp^{k-1}B_1 & B_2\\\hline
	O & U^{(k)}
	\end{array}\right),\quad B_1\in M_{dp^n}(\cO_{K_\wp}),\ B_2\in M_{dp^n,d(k-1)}(\cO_{K_\wp}).
	\]
	Corollary \ref{CorElDiv} and Proposition \ref{PropOrdAgree} (\ref{PropOrdAgree-1}) show that $U^{(k+p^n)}$ satisfies the assumptions of Proposition 
	\ref{PropPerturb}. Hence we obtain 
	$d(k+p^n,a)=d(V,a)$.
	By \cite[Lemma 2.3 (2)]{Ha_GVt}, the matrix $\wp^{k-1}B_1$ has no eigenvalue of slope less than $k-1$. 
	Since $a<k-1$, we also have $d(V,a)=d(k,a)$. This concludes the proof.
\end{proof}

\begin{thm}\label{ThmDMFP0}
	Suppose that $\frn\wp$ has a prime factor $\pi$ of degree one. Let $n\geq 1$ and $k\geq 2$ be any integers. 
	Let $\chi:\kappa(\wp)^\times\to \kappa(\wp)^\times$ be any character.
	Put $d=[\Gamma_1(\pi):\Gamma_1^\Theta(\frn,\wp^r)]$ and $\vep_\chi=d(k,\chi,0)$. 
	Let $a$ be any non-negative rational number satisfying 
	\[
	a<\min\{C(n,d,\vep_\chi),k-1\}.
	\]
	Then, for any integer $k'\geq k$, we have
	\[
	k'\equiv k \bmod p^n(q^d-1) \Rightarrow d(k',\chi,a)=d(k,\chi,a).
	\]
\end{thm}
\begin{proof}
	This follows in the same way as Theorem \ref{ThmDMFP}, using Proposition \ref{PropWindow0} and Proposition \ref{PropOrdAgree} (\ref{PropOrdAgree-chi}).
\end{proof}

\begin{thm}\label{ThmGM}
	Suppose that $\frn\wp$ has a prime factor $\pi$ of degree one. Let $n\geq 1$ and $k\geq 2$ be any integers and $a\leq n$ any non-negative rational number. 
	Put $d=[\Gamma_1(\pi):\Gamma_1^\Theta(\frn,\wp^r)]$ and $\vep=d(k,0)$. Suppose that the following conditions hold. 
	\begin{enumerate}
		\item\label{ThmGM-large} If $p=2$, then $n\geq 3$ or $d-\vep\leq 1$.
		\item\label{ThmGM-small} $2p^n> n(dn+2+d-2\vep)$.
	\end{enumerate}
	Then, for any integer $k'\geq k$, we have
	\[
	a<k-1,\ k'\equiv k \bmod p^n\Rightarrow d(k',a)=d(k,a).
	\]
	\end{thm}
\begin{proof}
	This follows in the same way as Theorem \ref{ThmDMFP}, using Proposition \ref{PropPerturbGM} instead of Proposition \ref{PropPerturb}.
	\end{proof}


It will be necessary to use an increasing function no more than $C(n,d,\vep)$ instead of itself. Here we give an example.

\begin{lem}\label{LemD}
	Let $n,d\geq 1$ and $\vep\geq 0$ be any integers satisfying $\vep\leq d$. Put 
	\begin{align*}
	D_2(n,d,\vep)&=\frac{1}{d}\left\{\sqrt{2dp^n+(d-\vep+1)(2d-\vep-1)}-\frac{3}{2}d+\vep\right\},\\
	D(n,d,\vep)&=\min\{C_1(n,d,\vep), D_2(n,d,\vep)\}.
	\end{align*}
	Then $D(n,d,\vep)$ is an increasing function of $n$ satisfying $D(n,d,\vep)\leq C(n,d,\vep)$.
	\end{lem}
\begin{proof}
	Since $C_1(n,d,\vep)$ is increasing for $n\geq 1$, it suffices to show $D_2(n,d,\vep)\leq C_2(n,d,\vep)$.
	Put $m=d-\vep+1$ and $x=dq_l+m\geq 1$. Since $r_l\in [0,d-1]$, for any $l>\vep$ we have
	\[
	\frac{2p^n+dq_l(q_l-1)+2q_l(r_l+1)}{2(l-\vep)}\geq\frac{2p^n+dq_l(q_l-1)+2q_l}{2x}.
	\]
	The right-hand side equals
	\begin{align*}
	&\frac{1}{2x}\left\{2p^n+d\left(\frac{x-m}{d}\right)\left(\frac{x-m}{d}-1\right)+2\left(\frac{x-m}{d}\right)\right\}\\
	&=\frac{x}{2d}+\frac{1}{2dx}\left(2dp^n+m(m+d-2)\right)-\frac{m}{d}-\frac{1}{2}+\frac{1}{d}.
	\end{align*}
	By the inequality of arithmetic and geometric means, it is no less than $D_2(n,d,\vep)$ and the lemma follows.
	\end{proof}

When $\frn=1$, $\wp=t$ and $r=1$, we have $\Gamma_1^\Theta(\frn,\wp^r)=\Gamma_1(t)$, $d=1$ and $\vep=1$ by \cite[Lemma 2.4]{Ha_GVt}, which yields
\[
C_1(n,1,1)=p^n\left(\frac{p^n+3}{2p^n+2}\right)\geq D_2(n,1,1)=\sqrt{2p^n}-\frac{1}{2}.
\]
Thus we obtain
\begin{equation}\label{EqnD11}
D(n,1,1)=\sqrt{2p^n}-\frac{1}{2}>0
\end{equation}
and Theorem \ref{ThmDMFP} gives the following improvement of \cite[Theorem 1.1]{Ha_GVt}.

\begin{cor}\label{CorGVt}
	Suppose $\frn=1$, $\wp=t$ and $r=1$. Let $k\geq 2$ be any integer and $a$ any non-negative rational number. 
	Let $n\geq 1$ be any integer satisfying 
	\[
	\frac{1}{2}\left(a+\frac{1}{2}\right)^2<p^n.
	\]
	Then, for any integer $k'\geq k$, we have
	\[
	a<k-1,\ k'\equiv k \bmod p^n\Rightarrow d(\Gamma_1(t),k',a)=d(\Gamma_1(t),k,a).
	\]
	\end{cor}




\section{$\wp$-adic continuous family}\label{SecBF}

We say $F\in \cV_k(\bC_\wp)$ is a Hecke eigenform if it is a non-zero eigenvector of $T_Q$ for any $Q\in A$. We denote by $\lambda_Q(F)$ the $T_Q$-eigenvalue of 
$F$.
Since Hecke operators commute with each other, if $d(k,a)=1$ (\textit{resp.} $d(k,\chi,a)=1$) then any non-zero $U$-eigenform in $\cV_k(\bC_\wp)$ 
(\textit{resp.} $\cV_k(\bC_\wp)(\chi)$) of slope $a$ 
is a Hecke eigenform.


\subsection{Construction of the family}\label{SubsecBF} Now we prove the following main theorem of this paper.

\begin{thm}\label{ThmBF}
	Suppose that $\frn\wp$ has a prime factor $\pi$ of degree one. Let $n\geq 1$ and $k_1\geq 2$ be any integers. 
	Put $d=[\Gamma_1(\pi):\Gamma_1^\Theta(\frn,\wp^r)]$ and $\vep=d(k_1,0)$. 
	Let $a$ be any non-negative rational number satisfying 
	\[
	a<\min\{C(n,d,\vep),k_1-1\}.
	\]
	Let $n'\geq 1$ be any integer satisfying 
	\[
	p^n-p^{n'}-a\geq 0,\quad a<C(n',d,\vep).
	\]
	
	Suppose $d(k_1,a)=1$. Let $F_1\in \cV_{k_1}(\bC_\wp)$ be a Hecke eigenform of slope $a$. Then, for any integer $k_2\geq 
	k_1$ satisfying
	\[
	k_2\equiv k_1\bmod p^n,
	\] 
	we have $d(k_2,a)=1$ and thus there exists a Hecke eigenform $F_2\in \cV_{k_2}(\bC_\wp)$ of 
	slope $a$ which is unique up to a scalar multiple. Moreover, for any $Q$ we have
	\begin{equation}\label{EqnThmBF}
	v_\wp(\lambda_Q(F_1)-\lambda_Q(F_2))> p^n-p^{n'}-a.
	\end{equation}
	\end{thm}
\begin{proof}
	By Proposition \ref{PropOrdAgree} (\ref{PropOrdAgree-1}), we may assume $(k_1,k_2)=(k,k+p^n)$ for some integer $k\geq 2$. Theorem \ref{ThmDMFP} yields $d(k
	+p^n,a)=1$ and any non-zero $U
	$-eigenform 
	$F_2\in \cV_{k_2}(\bC_\wp)$ of slope $a$ is a Hecke eigenform. 
	Take a finite extension $E/K_\wp$ inside $\bC_\wp$ 
containing $\lambda_Q(F_i)$ and $\lambda_\wp(F_i)$ for 
	$i=1,2$. We may assume $F_i\in \cV_{k_i}(\cO_E)$. We identify $\cV_{k_i}(\cO_E)$ with $\cO_E^{d(k_i-1)}$ via the ordered basis $
	\frB_{k_i}$. Then we can write
	\[
	F_2=\begin{pmatrix}
	x\\y
	\end{pmatrix}, \quad x\in\cO_E^{dp^n},\ y\in \cO_E^{d(k-1)},
	\]
	where each entry of $x$ is the coefficient of $v_{s,t}\in\frB_{k_2}$ in $F_2$ with $t<p^n$. 
	For any integer $N$ and $z={}^t\!(z_1,\ldots,z_N)\in \cO_E^N$, we put 
	\[
	v_\wp(z)=\min\{v_\wp(z_i)\mid i=1,\ldots,N\}. 
	\]
	Replacing $F_i$ by its scalar multiple, we may assume $v_\wp(F_i)=0$.
	
	For any $H\in \cV_{k_i}(\cO_E)$, we denote by $\bar{H}$ its image by the natural map $\cV_{k_i}(\cO_E)\to \cV_{k_i}(\cO_{E,p^n})$. 
	Consider the weight reduction map
	\[
	1\otimes \rho_{k,p^n}:\cV_{k+p^n}(\cO_{E,p^n})\to \cV_{k}(\cO_{E,p^n})
	\] 
	as in \S\ref{SubsecWtRed}, which we denote by $\rho$. Then $\rho(\bar{F}_2)=y \bmod \wp^{p^n}$. 
	
	We claim $v_\wp(y)\leq a$. Indeed, if $v_\wp(x)\geq v_\wp(y)$, then the assumption $v_\wp(F_2)=0$ yields $v_\wp(y)=0$. 
	If $v_\wp(x)< v_\wp(y)$, then $v_\wp(x)=0$ and Corollary \ref{CorWindow} gives
	\[
	\lambda_\wp(F_2) x=\wp^{k-1}B_1x+B_2y.
	\]
	Since $v_\wp(\lambda_\wp(F_2))=a<k-1$, this forces $v_\wp(y)\leq a$ and the claim follows.
	
	Take $G_1\in \cV_k(\cO_E)$ satisfying $\bar{G}_1=\rho(\bar{F}_2)$. By Lemma \ref{LemWtRed}, we have
	\begin{equation}\label{EqnRhoTQ}
	T_Q(G_1)\equiv \lambda_Q(F_2)G_1,\quad U(G_1)\equiv \lambda_\wp(F_2)G_1\bmod \wp^{p^n}\cV_k(\cO_E).
	\end{equation}
	Since we have $a<C(n,d,\vep)<p^n$, the above claim yields $v_\wp(G_1)\leq a$. If $G_1\in \cO_E F_1$, then $G_1$ is a 
	Hecke eigenform with the same eigenvalues as those of $F_1$. Thus we have
	\[
	\lambda_Q(F_1)\bar{G}_1=T_Q(\bar{G}_1)=\lambda_Q(F_2)\bar{G}_1,
	\]
	which gives 
	\begin{equation}\label{EqnCaseOne}
	v_\wp(\lambda_Q(F_1)-\lambda_Q(F_2)) \geq p^n-a.
	\end{equation}
	
	Suppose $G_1\notin \cO_E F_1$, and take $H_1\in \cV_k(\cO_E)$ such that $F_1$ and $H_1$ form a basis of a direct summand of $\cV_k(\cO_E)$ containing $G_1$. 
	Write 
	\begin{equation}\label{EqnGFHdef}
	G_1=\alpha F_1 +\beta H_1,\quad \alpha,\beta\in \cO_E. 
	\end{equation}
	Then $\beta\neq 0$. By (\ref{EqnRhoTQ}), for any $R\in \{\wp,Q\}$ we have
	\[
	\lambda_R(F_2) G_1\equiv T_R(G_1)=\alpha\lambda_R(F_1)F_1+\beta T_R(H_1) \bmod \wp^{p^n}\cV_k(\cO_E).
	\]
	Combined with (\ref{EqnGFHdef}), this implies
	\begin{equation}\label{EqnUH}
	\beta T_R(H_1)\equiv \alpha(\lambda_R(F_2)-\lambda_R(F_1)) F_1+\beta\lambda_R(F_2) H_1\bmod \wp^{p^n}\cV_k(\cO_E)
	\end{equation}
and thus we obtain
	\begin{equation}\label{EqnAQ}
	\alpha(\lambda_R(F_1)-\lambda_R(F_2))\equiv 0 \bmod (\beta, \wp^{p^n}).
	\end{equation}
	
	Put $b=v_\wp(\beta)$. Suppose $b> p^n-p^{n'}$. Since $v_\wp(F_1)=0$ and 
	\[
	v_\wp(G_1)\leq a\leq p^n-p^{n'}<b, 
	\]
	(\ref{EqnGFHdef}) gives $v_\wp(\alpha)\leq a$ and (\ref{EqnAQ}) yields
	\begin{equation}\label{EqnCaseTwo}
	v_\wp(\lambda_Q(F_1)-\lambda_Q(F_2))>  p^n-p^{n'}-a.
	\end{equation}
	
	Suppose $b\leq p^n-p^{n'}$. In this case we have $\beta^{-1}\wp^{p^n}\in\cO_E$ and by (\ref{EqnAQ}) we can write 
	\[
	\alpha(\lambda_\wp(F_2)-\lambda_\wp(F_1))=\beta\nu 
	\]
	with some 
	$\nu\in \cO_E$.
	Then (\ref{EqnUH}) shows
	\begin{equation}\label{EqnUFH}
	U(H_1)\equiv \nu F_1+\lambda_\wp(F_2) H_1\bmod \beta^{-1}\wp^{p^n}\cV_k(\cO_E).
	\end{equation}
	
	Take an ordered basis $(F_1,H_1,\tilde{v}_3,\ldots,\tilde{v}_{d(k-1)})$ of the $\cO_E$-module $\cV_k(\cO_E)$, and we denote by $\tilde{U}^{(k)}$
	the representing matrix of $U$ with respect to it. By (\ref{EqnUFH}), we can write
	\[
	\tilde{U}^{(k)}=\left(\begin{array}{cc|ccc}
	\lambda_\wp(F_1) & \nu  +\beta^{-1}\wp^{p^n} c_1& * &\cdots &*\\
	0 & \lambda_\wp(F_2) +\beta^{-1}\wp^{p^n} c_2 & * &\cdots &*\\
	0 & \beta^{-1}\wp^{p^n} c_3 & *&\cdots &*\\
	\vdots& \vdots & \vdots &\cdots &\vdots\\
	0 & \beta^{-1}\wp^{p^n} c_{d(k-1)} & *&\cdots &*
	\end{array}\right),\quad c_1,\ldots,c_{d(k-1)}\in \cO_E.
	\]
	Note that the elementary divisors of $\tilde{U}^{(k)}$ and 
	$U^{(k)}$ agree with each other. Let $V$ be the element of $M_{d(k-1)}(\cO_E)$ with the same columns as those of 
	$\tilde{U}^{(k)}$ except the second column which we require to be
	\[
	\begin{pmatrix}\nu \\ \lambda_\wp(F_2) \\ 0 \\\vdots \\0\end{pmatrix}.
	\]
	Then we have $d(V,a)\geq 2$. 
	On the other hand, since $p^n-b\geq p^{n'}$, 
	the assumption $a<C(n',d,\vep)$ and Proposition 
	\ref{PropPerturb}
	yield $d(V,a)=d(k,a)=1$, which is the contradiction. Thus the case $b\leq p^n-p^{n'}$ never occurs. Now the theorem follows from 
	(\ref{EqnCaseOne}) and (\ref{EqnCaseTwo}).
	\end{proof}

\begin{rmk}\label{RmkMain0}
	Putting $\vep=d(k_1,\chi,0)$ and assuming $d(k_1,\chi,a)=1$, the same proof using Proposition \ref{PropWindow0} and Theorem \ref{ThmDMFP0} shows that we can 
	construct, from a Hecke eigenform 
	$F_1\in \cV_{k_1}(\bC_\wp)
	(\chi)$ of slope $a$, a Hecke eigenform $F_2\in \cV_{k_2}(\bC_\wp)(\chi)$ of slope $a$ satisfying (\ref{EqnThmBF}) for any integer $k_2\geq 
	k_1$ with
	\[
	k_2\equiv k_1\bmod p^n(q^d-1).
	\] 
	\end{rmk}

\begin{proof}[Proof of Theorem \ref{ThmMainIntro}]
Suppose that $n$, $k$ and $a$ satisfy the assumptions of Theorem \ref{ThmMainIntro}. 
Take any $k'\geq k$ satisfying 
\[
m=v_p(k'-k)\geq \log_p(p^n+a).
\] 
Since $n\leq m$ and $D(n,d,\vep)$ is an increasing function of $n$ satisfying $D(n,d,\vep)\leq C(n,d,\vep)$, we have
\[
a<\min\{C(m,d,\vep),k-1\},\quad p^m-p^n-a\geq 0,\quad a<C(n,d,\vep).
\]

Note that, if $d(k,a)=1$, then any $U$-eigenform of slope $a$ in $\cV_k(\bC_\wp)$
is identified with a scalar multiple of that in $\cV_k(\Kbar)\subseteq S_k(\Gamma_1^\Theta(\frn,\wp^r))$ via the fixed embedding $\iota_\wp$.
Thus Theorem \ref{ThmBF} produces a Hecke eigenform $F_{k'}\in S_{k'}(\Gamma_1^\Theta(\frn,\wp^r))$
such that for any $Q$ we have
\[
v_\wp(\iota_\wp(\lambda_Q(F_{k'})-\lambda_Q(F_k)))> p^m-p^n-a.
\]
This concludes the proof of Theorem \ref{ThmMainIntro}.
\end{proof}



\subsection{Examples}\label{SubsecEx}

We assume $\frn=1$, $\wp=t$, $r=1$ and $\Gamma_1^\Theta(\frn,\wp^r)=\Gamma_1(t)$. In this case we have $d=1$ and $d(k,0)=1$ 
for any $k\geq 2$. 
In the following, we give examples of congruences between Hecke eigenvalues obtained by Theorem \ref{ThmMainIntro} for this case, using results of 
\cite{BV0,LM,Pet}. Note 
that the Hecke 
operator at $Q$ considered in \cite{BV0,Pet} is $QT_Q$ with our normalization.

\subsubsection{Slope zero forms}

By $d(k,0)=1$, any $U$-eigenform of slope zero in $S_k(\Gamma_1(t))$ is a member of a $t$-adic continuous family obtained by Theorem \ref{ThmMainIntro}.
Some of such eigenforms can be given by the theory of $A$-expansions \cite{Pet}. 

For any integer $k\geq 3$ satisfying $k \equiv 2\bmod q-1$, Petrov 
constructed an element $f_{k,1}\in S_k(\mathit{SL}_2(A))$ with $A$-expansion \cite[Theorem 1.3]{Pet}.
We know that $f_{k,1}$ is a Hecke eigenform whose Hecke eigenvalue at $Q$ is one for any $Q$;
this follows from a formula for the Hecke action \cite[p.~2252]{Pet} and $c_a=a^{k-n}$.

For such $k$, let $f^{(t)}_{k,1}\in S_k(\Gamma_1(t))$ be the $t$-stabilization of $f_{k,1}$ of finite slope, namely
\[
f^{(t)}_{k,1}(z)=f_{k,1}(z)-t^{k-1}f_{k,1}(tz).
\]
It is non-zero by \cite[Theorem 2.2]{Pet}. Moreover, we can show that $f^{(t)}_{k,1}$ is a Hecke eigenform which also satisfies $\lambda_Q(f^{(t)}_{k,1})=1$ for 
any $Q$.

\begin{prop}\label{PropOrdt}
	Let $k\geq 2$ be any integer and $F_k$ any non-zero element of $S_k(\Gamma_1(t))$ of slope zero. Then we have $\lambda_Q(F_k)=1$ for any $Q$.
	\end{prop}
\begin{proof}
	Let $r\in \{0,1,\ldots,q-2\}$ be an integer satisfying $k\equiv r\bmod q-1$. For $a=0$, we see from (\ref{EqnD11}) that the assumptions of Theorem 
	\ref{ThmMainIntro} are satisfied by
	$n=1$. Then, for any integer $s\geq 1$, we obtain a Hecke eigenform of slope zero
	\[
	F_{k'}\in S_{k'}(\Gamma_1(t)),\quad k'=k+(q+1-r)q^{s}
	\]
	such that, with the fixed embedding $\iota_t:\bar{K}\to \bC_t$, we have
	\[
	\iota_t(\lambda_Q(F_{k'}))\equiv \iota_t(\lambda_Q(F_{k})) \bmod t^{q^s-p}\quad \text{ for any }Q.
	\]
	Since $k'\geq 3$, $k'\equiv 2 \bmod q-1$ and $d(k',0)=1$, we see that $F_{k'}$ is a scalar multiple of $f^{(t)}_{k',1}$ and thus $\lambda_Q(F_{k'})=1$. 
	Since 
	$s$ 
	is 
	arbitrary, this implies $\lambda_Q(F_k)=1$.
	\end{proof}

\begin{cor}\label{CorHidaTriv}
	Let $k\geq 2$ and $r\geq 1$ be any integers. Then there exists a unique character $\chi:\kappa(\wp)^\times\to \kappa(\wp)^\times$ satisfying 
	$d(\Gamma_0^p(t^r),k,\chi,0)\neq 0$. For such $\chi$, we have $d(\Gamma_0^p(t^r),k,\chi,0)=1$ and any Hecke eigenform $F$ of slope zero in 
	$S_k(\Gamma_0^p(t^r))(\chi)$ 
	satisfies 
	$
	\lambda_Q(F)=1$ for any $Q$.
	\end{cor}
\begin{proof}
	Since $\Gamma_0^p(t)=\Gamma_1(t)$, Proposition \ref{PropHida} implies $d(\Gamma_0^p(t^r),k,0)=1$. Since we have
	\[
	d(\Gamma_0^p(t^r),k,0)=\sum_\chi d(\Gamma_0^p(t^r),k,\chi,0),
	\]
	the uniqueness of $\chi$ and the assertion on the dimension follow. Let $F_k$ be any Hecke eigenform of slope zero in $S_k(\Gamma_1(t))$. Since the natural 
	inclusion $S_k(\Gamma_1(t))\to S_k(\Gamma_0^p(t^r))$ is compatible with Hecke operators, $F$ is a scalar multiple of the image of $F_k$. Hence the last 
	assertion follows from Proposition \ref{PropOrdt}.
	\end{proof}


\begin{rmk}\label{RmkHida}
	Note that, since the only $p$-power root of unity in $\bC_\wp$ is one, there exists no non-trivial finite order character 
	$1+\wp\cO_{K_\wp}\to \bC_{\wp}^\times$. Thus it seems to the author that, if we try to generalize Hida theory including \cite[\S7.3, Theorem 3]{Hida_E}
	to Drinfeld cuspforms of level $\Gamma_1(t^r)$, then it would be natural to restrict ourselves to those of level $\Gamma_0^p(t^r)$. However, 
	Corollary \ref{CorHidaTriv} shows that such a generalization is trivial.
	\end{rmk}


\subsubsection{Slope one forms}

Let us consider the case $p=q=3$ and $a=1$. Since $D(1,1,1)=\sqrt{6}-\frac{1}{2}=1.949\ldots$, the assumptions of Theorem \ref{ThmMainIntro} are 
satisfied by $k\geq 3$ and $n=1$. Then a computation using \cite[(17)]{BV0} shows $d(10,1)=1$. 
Let $G_{10}$ and $G_{19}$ be any non-zero Drinfeld cuspforms of level $\Gamma_1(t)$ and slope one in weights $10$ and $19$, respectively.
Then Theorem \ref{ThmMainIntro} gives
\begin{equation}\label{EqnEx1}
v_t(\iota_t(\lambda_Q(G_{10})-\lambda_Q(G_{19})))>5
\end{equation}
for any $Q$.

For $Q=t$, using \cite[(17)]{BV0} we can show that $\lambda_t(G_{10})=-t-t^3$, and $\lambda_t(G_{19})$ is a root of the polynomial
\begin{align*}
X^4&+(t+t^3)X^3+(-t^{8}+t^{10}+t^{12}+t^{14}+t^{16})X^2\\
&+(-t^{9}-t^{11}+t^{13}+t^{15}+t^{17}+t^{19})X+(-t^{18}-t^{20}+t^{24}+t^{26}+t^{28})
\end{align*}
(see also \cite{Val_Table}).
Put $\iota_t(\lambda_t(G_{19}))=ty$ with $v_t(y)=0$. Then we obtain $y^3(y+1+t^2)\equiv 0\bmod t^6$ and $\iota_t(\lambda_t(G_{10}))\equiv 
\iota_t(\lambda_t(G_{19}))\bmod t^7$, 
which satisfies (\ref{EqnEx1}).
In fact, plugging in $X=-t-t^3+Z$ to the polynomial above yields $v_t(\iota_t(\lambda_t(G_{10})-\lambda_t(G_{19})))=9$. 

We identify $S_k(\Gamma_1(t))$ with $\bC_\infty^{k-1}$ via the ordered basis
\[
\{\mathbf{c}_j(\gamma_0)=\mathbf{c}_j(\bar{e})\mid 0\leq j\leq k-2\}
\]
defined in \cite{LM,BV0}. Then $G_{10}$ is identified with the vector
\[
{}^t\!\left(0,1+t^2,0,-(1+t^2),0,-t^2,0,1,0\right).
\]
Thus $\lambda_{1+t}(G_{10})$ agrees with the evaluation $T_{1+t}(G_{10})(\gamma_0)(X^7Y)$ after identifying $G_{10}$ with a harmonic cocycle.
By \cite[(7.1)]{LM}, we have $\lambda_{1+t}(G_{10})=1-t-t^3$. On the other hand, by computing the characteristic polynomial of $T_{1+t}$ 
acting on $S_{19}(\Gamma_1(t))$ using \cite[(7.1)]{LM} and plugging in $X=1-t-t^3+Z$ into it, (\ref{EqnEx1}) implies $v_t(\iota_t(\lambda_{1+t}(G_{10})-
\lambda_{1+t}(G_{19})))=9$. 

Note that, since these eigenvalues 
are not powers of $t$ or $1+t$, the Hecke eigenforms $G_{10}$ and $G_{19}$ are not the $t$-stabilizations of Hecke eigenforms with $A$-expansion. 



\end{document}